\documentclass[12pt, letterpaper]{amsart}

\usepackage[margin=1in]{geometry}
\usepackage{amscd}
\usepackage{amssymb}
\usepackage{amsmath}
\usepackage{amsthm}
\usepackage{enumerate}
\usepackage{tikz}
\usepackage{float}
\usetikzlibrary{arrows}
\usepackage[font=small,labelfont=bf]{caption}
\usepackage{xcolor}
\usepackage[normalem]{ulem}
\usepackage{mathtools}
\usepackage[colorlinks=true, linkcolor=blue, citecolor=blue,
pagebackref=true]{hyperref}
\usepackage{fouriernc}

\newtheorem{theorem}{Theorem}

\newtheorem*{theoremY*}{Theorem Y}

\newtheorem{lemma}{Lemma}

\newtheorem*{claim*}{Claim}
\newtheorem{conjecture}{Conjecture}

\theoremstyle{definition}
\theoremstyle{remark}

\newtheorem*{remark*}{Remark}

\theoremstyle{definition}
\newtheorem{question}{Question}

\renewcommand{\Bbb}[1]{\mathbb{#1}}
\newcommand{\N}{{\Bbb N}}         
\newcommand{\I}{{\Bbb I}}
\newcommand{\R}{{\Bbb R}}        
\newcommand{\Z}{{\Bbb Z}}         

\newcommand{\cA}{{\mathcal A}}

\newcommand{\cH}{{\mathcal H}}

\newcommand{\cM}{{\mathcal M}}

\newcommand{\x}{\mathbf{x}}
\newcommand{\X}{\mathbf{X}}
\newcommand{\y}{\mathbf{y}}
\newcommand{\p}{\mathbf{p}}
\newcommand{\q}{\mathbf{q}}

\newcommand{\0}{\mathbf{0}}

\newcommand{\NN}{\mathbb{N}}

\newcommand{\RR}{\mathbb{R}}

\newcommand{\ZZ}{\mathbb{Z}}

\newcommand{\br}{\mathbf r}

\newcommand{\bz}{\mathbf z}

\newcommand{\bx}{\mathbf x}

\DeclarePairedDelimiter{\abs}{\lvert}{\rvert}
\DeclarePairedDelimiter{\set}{\lbrace}{\rbrace}
\DeclarePairedDelimiter{\floor}{\lfloor}{\rfloor}
\DeclarePairedDelimiter{\ceil}{\lceil}{\rceil}
\DeclarePairedDelimiter{\parens}{\lparen}{\rparen}
\DeclarePairedDelimiter{\brackets}{\lbrack}{\rbrack}

\def\eps{{\varepsilon}}

\title[Inhomogeneous approximation for systems of linear
forms]{Inhomogeneous approximation for systems of linear forms with
  primitivity constraints}

\author{Demi Allen}

\author{Felipe A.~Ram{\'i}rez}

\address{D.~Allen: Department of Mathematics and Statistics,  University of Exeter, Harrison Building, North
Park Road, Exeter, EX4 4QF, UK.}

\address{F.~A.~Ramirez: Department of Mathematics and Computer Science, Wesleyan
University, 265 Church Street, Middletown, CT 06459, USA.}

\makeatletter
\@namedef{subjclassname@2020}{\textup{2020} Mathematics Subject Classification}
\makeatother



\begin{document}
\frenchspacing

\maketitle

\begin{abstract}
  We study (inhomogeneous) approximation for systems of linear forms
  using integer points which satisfy additional primitivity
  constraints. The first family of primitivity constraints we consider
  were introduced in 2015 by Dani, Laurent, and Nogueira, and are
  associated to partitions of the coordinate directions. Our results
  in this setting strengthen a theorem of Dani, Laurent, and Nogueira,
  and address problems posed by those same authors. The second
  primitivity constraints we consider are analogues of the coprimality
  required in the higher-dimensional Duffin--Schaeffer conjecture,
  posed by Sprind\v{z}uk in the 1970's and proved by Pollington and
  Vaughan in 1990. Here, with attention restricted to systems of
  linear forms in at least three variables, we prove a univariate
  inhomogeneous version of the Duffin--Schaeffer conjecture for
  systems of linear forms, the multivariate homogeneous version of
  which was stated by Beresnevich, Bernik, Dodson, and Velani in 2009
  and recently proved by the second author.
\end{abstract}

\setcounter{tocdepth}{1} 
\tableofcontents

\section{Introduction}
\label{sec:introduction}

We are concerned here with the problem of determining whether for a
given sequence $(B_q)_{q=1}^\infty$ of balls in $\RR^m$ and a typical
$\x\in\operatorname{Mat}_{n\times m}(\RR)$ there are infinitely many
$(\p,\q)\in\ZZ^m\times\ZZ^n$ such that
\begin{equation}\label{eq:target}
  \q\x - \p \in B_{\abs{\q}},
\end{equation}
where $\abs{\q} =\max\{|q_1|,\dots,|q_n|\}$ denotes the maximum norm,
subject to the integer vectors $(\p,\q)$ satisfying additional
primitivity conditions. We consider two different kinds of conditions:
the first were introduced in work of Dani, Laurent, and
Nogueira~\cite{DLN}, and the second are relevant to a version of the
Duffin--Schaeffer conjecture for systems of linear forms.

The first type of primitivity conditions we consider---those
introduced by Dani, Laurent, and Nogueira---are associated to
partitions $\pi$ of $\set{1, 2, \dots, m + n}$. They are natural
refinements of the condition $\gcd(\p,\q)=1$ which arises often in the
classical setting. For an integer vector $\q \in \Z^n$ we write
$\gcd(\q)$ to mean $\gcd(q_1,\dots, q_n)$. In this first setting
(which is described in more detail in
Section~\ref{sec:partition-reduction}), we consider the set
$P(\pi)\subset \ZZ^m\times\ZZ^n$ consisting of points $(\p,\q)$ whose
coordinates corresponding to each partition element of $\pi$ are coprime. We
seek solutions to~(\ref{eq:target}) with $(\p,\q)\in P(\pi)$. In this
setting, Dani, Laurent, and Nogueira proved a doubly-metric
Khintchine--Groshev-type theorem~\cite[Theorem~1.1]{DLN} with certain
assumptions on the partition $\pi$ and a monotonicity assumption on
the sequence of the radii of the balls~$(B_q)_{q=1}^{\infty}$.

\theoremstyle{plain}
\newtheorem*{theorem*}{Theorem}
\begin{theorem*}[Dani--Laurent--Nogueira~\cite{DLN}]
  Let $m,n\in\NN$ and suppose $\pi = \set{\pi_1, \dots, \pi_k}$ is a
  partition of $\set{1, \dots, m+n}$ with $\abs{\pi_j}\geq m+1$ for
  each $j=1,\dots,k$. (Here, the notation $|\pi_j|$ denotes the number
  of elements in the partition component $\pi_j$.) If
  $\psi:\NN\to\RR_{\geq 0}$ is a mapping such that
  $x\mapsto x^{n-1}\psi(x)^m$ is non-increasing and the series
  \begin{equation}\label{eq:6}
    \sum_{q=1}^\infty q^{n-1}\psi(q)^m
  \end{equation}
  diverges, then for almost every pair
  $(\x,\y)\in \operatorname{Mat}_{n\times m}(\RR)\times \RR^m$, there
  exist infinitely many points $(\p,\q)\in P(\pi)$ such that
  \begin{equation}\label{eq:3}
    \abs{\q\x - \p-\y} < \psi(\abs{\q}).
  \end{equation}
  Conversely, if~(\ref{eq:6}) converges, then for almost every
  pair $(\x,\y)\in \operatorname{Mat}_{n\times m}(\RR)\times \RR^m$
  there are only finitely many $(\p,\q)\in P(\pi)$ such
  that~(\ref{eq:3}) holds.
\end{theorem*}

\begin{remark*}
  To phrase this in terms of~(\ref{eq:target}), take the balls
  $B_{|\q|}$ to be centered at $\y \in\RR^m$ with radius
  $\psi(\abs{\q})$.
\end{remark*}

\begin{remark*}
  In the case of the trivial partition
  $\pi = \set*{\set{1, 2, \dots, m+n}}$ this theorem is classical. For
  a discussion of the state of the art regarding the classical case,
  we refer the reader to \cite{AR, BV} and references therein. The
  above theorem, and also Theorems \ref{cor:monotone} and
  \ref{cor:nonmonotone} which follow, can be viewed as refinements of
  the classical setting allowing for non-trivial partitions.
\end{remark*}

Regarding the assumptions in this theorem, Dani, Laurent, and Nogueira
ask the following questions~\cite{DLN, L}:

\begin{question}[\cite{DLN}/{\cite[Problem
    1]{L}}] \label{question:singlymetric} Can the result be made
  singly-metric? It is expected that this result should hold for every
  fixed $\y \in \R^m$ and with almost every
  $\x \in \operatorname{Mat}_{n\times m}(\RR)$, as it does in the
  setting without any primitivity constraints, that is, in the
  classical setting of the inhomogeneous Khintchine--Groshev
  theorem (see, for example, \cite{S}).
\end{question}

\begin{question}[{\cite[Problem 2]{L}}] \label{question:partition} Can
  the assumptions on the partition $\pi$ be relaxed? In particular,
  can the assumption that each partition element have size at least
  $m+1$ be removed? Note that each partition element having size at
  least $2$ is a necessary condition, for otherwise $P(\pi)$ would lie
  in a proper hyperplane.
\end{question}

\begin{question}[{\cite[Problem 2]{L}}] \label{question:monotonicity}
  Can the monotonicity condition imposed on the approximating function
  $\psi$ be removed or relaxed? Monotonicity of $\psi$ is needed in the
  inhomogeneous Khintchine--Groshev theorem in the case
  $(m,n)=(1,1)$~\cite{DS}, and it is not needed when
  $nm>2$~\cite{AR}. The cases when $(m,n)= (2,1)$ or $(1,2)$ are open.
\end{question}

\section{Results}

\subsection{Main results}
The following theorem addresses Questions \ref{question:singlymetric}
and \ref{question:partition}. In particular, a singly metric version
of the theorem above due to Dani, Laurent and Nogueira holds without
any assumptions on the partition.

\begin{theorem}\label{cor:monotone}
  Let $m,n\in\NN$ and fix $\y\in\RR^m$. Suppose
  $\pi = \set{\pi_1, \dots, \pi_k}$ is a partition of
  $\set{1, \dots, m+n}$ with $\abs{\pi_j}\geq 2$ for each
  $j = 1,\dots,k$. If $\psi:\NN\to\RR_{\geq 0}$ is non-increasing
  and $\sum q^{n-1}\psi(q)^m$ diverges, then for almost every
  $\x\in \operatorname{Mat}_{n\times m}(\RR)$ there exist infinitely
  many points $(\p,\q)\in P(\pi)$ such that~(\ref{eq:3}) holds.
  
  Conversely, if $\sum q^{n-1}\psi(q)^m$ converges, then for almost
  every $\x\in \operatorname{Mat}_{n\times m}(\RR)$ there are only
  finitely many $(\p,\q)\in P(\pi)$ such that~(\ref{eq:3}) holds.
\end{theorem}

\begin{remark*}
  In fact, we prove a stronger theorem (Theorem~\ref{thm:monotone})
  where $\y$ can depend on $\abs{\q}$, that is, the target balls
  $B_{\abs{\q}}$ do not have to be concentric.
\end{remark*}

The following theorem shows that, further to
Theorem~\ref{cor:monotone}, we can also answer Question
\ref{question:monotonicity} affirmatively in the cases where $nm>2$
(mirroring the current knowledge in the classical setting) if we are
willing to impose a mild assumption on the partition.

\begin{theorem}\label{cor:nonmonotone}
  Let $m,n\in\NN$ be such that $nm>2$ and fix $\y\in\RR^m$. Suppose
  $\pi = \set{\pi_1, \dots, \pi_k}$ is a partition of
  $\set{1, \dots, m+n}$ such that $\abs{\pi_j}\geq 2$ for each
  $j=1,\dots,k$. Furthermore, suppose that there exists some
  $\ell \in \{1,\dots,k\}$ for which $\abs{\pi_\ell}\geq 3$ and
  $\pi_\ell\cap \set{m+1, \dots, m+n}\neq \emptyset$. If
  \mbox{$\psi:\NN\to\RR_{\geq 0}$} is any function such that
  $\sum q^{n-1}\psi(q)^m$ diverges, then for almost every
  $\x\in \operatorname{Mat}_{n\times m}(\RR)$ there exist infinitely
  many points $(\p,\q)\in P(\pi)$ such that~(\ref{eq:3})
  holds. 
  
  Conversely, if $\sum q^{n-1}\psi(q)^m$ converges, then for
  almost every $\x\in \operatorname{Mat}_{n\times m}(\RR)$ there are
  only finitely many $(\p,\q)\in P(\pi)$ such that~(\ref{eq:3}) holds.
\end{theorem}

\begin{remark*}
  As with the previous result, this one also follows from a stronger
  statement (Theorem~\ref{thm:nonmonotone}) where the target ball's center can
  move.
\end{remark*}

Next, we turn our attention to the following univariate inhomogeneous
analogue of the Duffin--Schaeffer conjecture for systems of linear
forms. A {\itshape{homogenenous multivariate}} Duffin--Schaeffer
conjecture for systems of linear forms, i.e. where $\psi$ is a
multivariate function depending on $\q$ rather than $\abs{\q}$, and
$\y=0$, was posed in~\cite{BBDV} and has recently been proved by the
second author in \cite{R}. We complement this recent work with the
following {\itshape{inhomogeneous univariate}} statement. Of course,
the univariate case is a special case of the multivariate case. The
novelty of the following statement therefore is the inhomogeneity
which is allowed.

\begin{theorem}[Univariate inhomogeneous Duffin--Schaeffer
  conjecture for systems of linear forms]\label{cor:idsc}
  Let $m,n\in\NN$ with $n>2$ and fix $\y\in\RR^m$. If
  $\psi:\NN\to\RR_{\geq 0}$ is a function such that
  \begin{equation}\label{eq:dscsum}
    \sum_{\q \in \ZZ^n} \parens*{\frac{\varphi(\gcd(\q))\psi(\abs{\q})}{\gcd(\q)}}^m= \infty,
  \end{equation}
  then for almost every $\x\in \operatorname{Mat}_{n\times m}(\RR)$
  there exist infinitely many points $(\p,\q)\in\ZZ^m\times\ZZ^n$ with
  $\gcd(p_i, \q)=1$ for every $i=1, \dots, m$ and such
  that~(\ref{eq:3}) holds. 
\end{theorem}

\begin{remark*}
  Again, this follows from a more general statement
  (Theorem~\ref{thm:idsc}) where $\y$ may vary.
\end{remark*}

We conjecture that the result holds without the condition on $n$.

\begin{conjecture}\label{conj:idsc}
  Theorem~\ref{cor:idsc} also holds when $n \leq 2$. 
\end{conjecture}

In the case when $m=n=1$ and $\y=0$, Conjecture~\ref{conj:idsc} is
exactly the Duffin--Schaeffer conjecture~\cite{DS}, which was proved
in 2020 by Koukoulopoulos and Maynard~\cite{KM}. In 1990, Pollington
and Vaughan~\cite{PV} proved Conjecture~\ref{conj:idsc} in the cases $(m,1)$
with $m\geq 2$ and $\y = 0$, thus verifying a higher-dimensional
simultaneous version of the classical Duffin--Schaeffer Conjecture as
postulated by Sprind\v{z}uk~\cite{S}. 

\subsection{Hausdorff measure statements}

In Diophantine Approximation, in addition to considering Lebesgue
measure, one is often interested in studying the Hausdorff measure and
dimension of sets. This is particularly pertinent for sets which have
zero Lebesgue measure as Hausdorff measures and dimensions can often
provide a means for distinguishing such sets. For example, to observe
this phenomenon, one can compare the classical works of Khintchine
\cite{KT}, Jarn\'{\i}k \cite{J1,J2}, and Besicovitch \cite{B}. For
definitions of Hausdorff measures and dimension, we refer the reader
to \cite{F}.

Below we record Hausdorff measure analogues of Theorems
\ref{cor:monotone}, \ref{cor:nonmonotone}, and \ref{cor:idsc}. For the
Hausdorff measure analogues of Theorems \ref{cor:monotone} and
\ref{cor:nonmonotone}, the statements we give follow immediately from
\cite[Theorem 7]{AB}, which itself is deduced from the mass
transference principle for systems of linear forms proved in
\cite[Theorem 1]{AB}. Given a function $\psi: \N \to \R_{\geq 0}$, a
partition $\pi = \{\pi_1,\pi_2,\dots,\pi_k\}$ of $[m+n]$ with
$|\pi_j| \geq 2$ for each $j=1,\dots,k$, fixed $\Phi \in \I^{mm}$,
and $\y \in \I^m$, define $\cM_{n,m}^{\pi, \y, \Phi}(\psi)$ to be the
set of $\x \in \I^{nm}$ such that
\[|\q\x-\p\Phi-\y|<\psi(|\q|)\] holds for $(\p,\q) \in P(\pi)$ with
arbitrarily large $|\q|$. When $\Phi$ is the $m \times m$ identity
matrix, we will omit the superscript $\Phi$ and simply write
$\cM_{n,m}^{\pi, \y}(\psi)$.

Throughout, we define a {\itshape dimension function} to be a
continuous function $f: \R_{\geq 0} \to \R_{\geq 0}$ with $f(r) \to 0$
as $r \to 0$. For a subset $X \subset \I^{nm}$, we denote by $|X|$ its
Lebesgue measure and by $\cH^f(X)$ its Hausdorff $f$-measure.

\begin{theorem*}[\cite{AB}]
  Let $\psi: \N \to \R_{\geq 0}$ be such that
  $\frac{\psi(q)}{q} \to 0$ as $q \to \infty$. Let
  $\pi = \{\pi_1,\pi_2,\dots,\pi_k\}$ be a partition of $[m+n]$ with
  $|\pi_j| \geq 2$ for each $j=1, \dots, k$ and let $\Phi \in \I^{mm}$
  and $\y \in \I^m$ be fixed. Let $f: \R_{\geq 0} \to \R_{\geq 0}$ be
  a dimension function such that $r^{-nm}f(r)$ is monotonic and
  $g: r \mapsto g(r) = r^{-m(n-1)}f(r)$ is also a dimension
  function. Define $\theta: \N \to \R_{\geq 0}$ by
\[\theta(q) = qg\left(\frac{\psi(q)}{q}\right)^{\frac{1}{m}}.\]
Then
\[\left|\cM_{n,m}^{\pi, \y, \Phi}(\theta)\right|=1 \qquad \text{implies} \qquad \cH^f\left(\cM_{n,m}^{\pi, \y, \Phi}(\psi)\right) = \cH^f(\I^{nm}).\]
\end{theorem*}

Applying the above theorem gives rise to the following Hausdorff measure versions of Theorems \ref{cor:monotone} and \ref{cor:nonmonotone}. We note that we may assume without loss of generality that $\frac{\psi(q)}{q} \to 0$ as $q \to \infty$ in Theorems \ref{cor:monotone}, \ref{cor:nonmonotone}, and \ref{cor:idsc}. Hence the appearance of this condition in the statements below is not restrictive.

\begin{theorem}[Hausdorff measure version of Theorem
  \ref{cor:monotone}] \label{Hausdorff monotone} Let $m,n \in \N$ and
  fix $\y \in \R^m$. Let $\pi = \{\pi_1,\pi_2,\dots,\pi_k\}$ be a
  partition of $[m+n]$ with $|\pi_j| \geq 2$ for each $j=1,\dots,k$
  and suppose that $\psi: \N \to \R_{\geq 0}$ is non-increasing (in
  particular, note that this means that $\frac{\psi(q)}{q} \to 0$ as
  $q \to \infty$). Let $f: \R_{\geq 0} \to \R_{\geq 0}$ be a dimension
  function such that $r^{-nm}f(r)$ is monotonic and
  $g: r \mapsto g(r) = r^{-m(n-1)}f(r)$ is also a dimension
  function. If
\[\sum_{q=1}^{\infty}{q^{n+m-1}g\left(\frac{\psi(q)}{q}\right)} = \infty, \]
then
\[\cH^f\left(\cM_{n,m}^{\pi, \y}(\psi)\right) = \cH^f(\I^{nm}).\]
\end{theorem}

\begin{theorem}[Hausdorff measure version of Theorem
  \ref{cor:nonmonotone}] \label{Hausdorff nonmonotone} Let
  $m,n \in \N$ and fix $\y \in \R^m$. Let
  $\pi = \{\pi_1,\pi_2,\dots,\pi_k\}$ be a partition of $[m+n]$ with
  $|\pi_j| \geq 2$ for each $j=1,\dots,k$ and further suppose that
  there exists some $\ell \in \{1,\dots,k\}$ such that
  $|\pi_{\ell}| \geq 3$ and
  $\pi_{\ell} \cap \{m,m+1,\dots,m+n\} \neq \emptyset$. Suppose that
  $\psi: \N \to \R_{\geq 0}$ is such that $\frac{\psi(q)}{q} \to 0$ as
  $q \to \infty$ and suppose $f: \R_{\geq 0} \to \R_{\geq 0}$ is a
  dimension function such that $r^{-nm}f(r)$ is monotonic and
  $g: r \mapsto g(r) = r^{-m(n-1)}f(r)$ is also a dimension
  function. If
\[\sum_{q=1}^{\infty}{q^{n+m-1}g\left(\frac{\psi(q)}{q}\right)} = \infty, \]
then
\[\cH^f\left(\cM_{n,m}^{\pi, \y}(\psi)\right) = \cH^f(\I^{nm}).\]
\end{theorem}

\begin{remark*}
Theorems \ref{cor:monotone} and \ref{cor:nonmonotone} further refine some similar statements proved in \cite{AB}, see Theorems 8--10 therein.
\end{remark*}

Given a function $\psi: \N \to \R_{\geq 0}$ and a fixed $\y \in \R^m$, let us now denote by $\cA_{n,m}^{\y}(\psi)$ the set of points $\x \in \I^{nm}$ for which 
\[|\q \x-\p-\y| < \psi(|\q|)\]
for infinitely many pairs of vectors $(\p,\q) \in \Z^m \times \Z^n \setminus \{\0\}$ with $\gcd(p_i,\q)=1$ for every $1 \leq i \leq m$.  
By modifying arguments contained in \cite[Section 2]{AB}, combining \cite[Theorem 1]{AB} with Theorem \ref{cor:idsc}, it is possible to obtain the following inhomogeneous univariate Hausdorff measure statement.

\begin{theorem}[Hausdorff measure version of Theorem \ref{cor:idsc}] \label{Hausdorff IDSC}
Suppose $\psi: \N \to \R_{\geq 0}$ is any function such that $\frac{\psi(q)}{q} \to 0$ as $q \to \infty$ and suppose $\y \in \R^m$ is fixed. Let $f: \R_{\geq 0} \to \R_{\geq 0}$ be a dimension function such that $r^{-nm}f(r)$ is monotonic and $g: r \mapsto g(r) = r^{-m(n-1)}f(r)$ is also a dimension function. If
\begin{equation}
    \sum_{\q \in \ZZ^n \setminus \{0\}} \parens*{\frac{\varphi(\gcd(\q))}{\gcd(\q)}\abs{\q}}^mg\parens*{\frac{\psi(|\q|)}{|\q|}}= \infty,
  \end{equation}
then
\[\cH^f(\cA_{n,m}^{\y}(\psi)) = \cH^f(\I^{nm}).\]
\end{theorem}

\begin{remark*}
  The complementary convergence statements corresponding to Theorems
  \ref{Hausdorff monotone}, \ref{Hausdorff nonmonotone},
  and~\ref{Hausdorff IDSC} can all be proved via standard covering
  arguments. Moreover, in the convergence cases, no monotonicity
  assumptions are required.
\end{remark*}

\section{Partition reduction}\label{sec:partition-reduction}

Fix $m,n\in\NN$. We are concerned with partitions of
$[m+n] := \set{1, 2, \dots, m, m+1, \dots, m+n}$. For example, $[m+n]$
itself can be regarded as the trivial partition, more correctly
written as $\set{[m+n]}$. Another important partition is
$\set{[m], m + [n]}$, where
$m + [n] = m + \set{1, \dots, n} = \set{m+1, \dots, m + n}$.

Let $\pi = (\pi_1, \dots, \pi_k)$ be a partition of $[m + n]$ such
that $\abs{\pi_j} \geq 2$ for each $j=1, \dots, k$.  By re-ordering
the partition components if necessary, we may suppose that there exist
$0\leq a \leq b \leq k$ such that
\begin{itemize}
\item For each $j \in [1,a]\cap \Z$ we have $\pi_j\cap [m]\neq \emptyset$ and
  $\pi_j\cap m+[n]\neq \emptyset$.
\item For each $j\in (a, b] \cap \Z$ we have $\pi_j \subset [m]$.
\item For each $j\in (b, k]\cap \Z$ we have $\pi_j \subset m + [n]$.
\end{itemize}
For each $j = 1, \dots, k$ it is convenient to abuse notation and let
$\pi_j$ also denote the projection of $\ZZ^{m+n}$ onto the coordinates
corresponding to $\pi_j$, regarded as a vector in
$\ZZ^{\abs{\pi_j}}$. (The context will disambiguate.)  For example, if
we use $\pi_1 = [m]$ and $\pi_2 = m + [n]$, and elements of
$\ZZ^{m + n}$ are written $(\p, \q)$ with $\p \in \ZZ^m$ and
$\q \in \ZZ^n$, then $\pi_1(\p,\q) = \p$ and $\pi_2(\p,\q) = \q$.

Indeed, let us keep the convention of writing $(\p,\q)$ for elements
of $\ZZ^{m+n} = \ZZ^m\times\ZZ^n$. Then
\begin{equation*}
  P(\pi) = \set*{(\p,\q)\in\ZZ^{m+n} : \text{ for each } j=1,\dots, k, \quad \gcd(\pi_j(\p,\q)) = 1}.
\end{equation*}
Note that $P(\set{[m+n]})$ is the set of primitive points in
$\ZZ^{m+n}$. For $\q \in \Z^n$, let
\begin{equation*}
  P(\pi,\q) = \set*{\p \in\ZZ^m : (\p,\q)\in P(\pi)}.
\end{equation*}
For example, $P(\set{[m+n]}, \q) = \ZZ^{m+n}$ if $\gcd(\q)=1$, but in
general $P(\set{[m+n]}, \q)$ is a proper subset of $\ZZ^{m+n}$. Finally, let
\begin{equation*}
  Q(\pi) = \set{\q \in \ZZ^n : P(\pi, \q)\neq \emptyset}.
\end{equation*}
These are the vectors $\q$ such that for all $j\in(b,k]\cap \Z$,
$\gcd(\pi_j(\0, \q)) = 1$, with $\0$ denoting the $m$-dimensional
$0$-vector.

\section{Arithmetic lemmas}
\label{sec:arithmetic}

In this section we collect some useful arithmetic sums, many of which
are well-known. The main purpose of this section is to prove
Lemma~\ref{lem:funny}, a sum which is crucial later in the proof of
Lemma~\ref{lem:counting}.

We will often use the Vinogradov symbols. Suppose that
$f: \R_{\geq 0} \to \R_{> 0}$ and $g: \R_{\geq 0} \to \R_{>0}$ are
functions. We say that $f \ll g$ if there exists a constant $C>0$ such
that $f(x) \leq Cg(x)$ for all $x \in \R_{\geq 0}$. If $f \ll g$ and
$g \ll f$, we write $f \asymp g$ and say that $f$ and $g$ are
{\itshape comparable}. We write $f \sim g$ if
$\frac{f(x)}{g(x)} \to 1$ as $x \to \infty$ and we write $f = o(g)$ if
$\frac{f(x)}{g(x)} \to 0$ as $x \to \infty$.

Throughout, we use $\varphi$ to denote the {\itshape Euler totient
  function} and $\mu$ to denote the {\itshape M\"{o}bius
  function}. For definitions and properties, the reader is referred to
\cite{HW}.

\begin{lemma}\label{lem:primitiveball}
  For each integer $D\geq 2$ there exists a constant $C>0$ such that
  \begin{equation*}
    \sum_{\substack{\q\in\ZZ^D \\ \abs{\q} \leq q \\ \gcd(\q)=1}} 1 \geq Cq^D
  \end{equation*}
  for all $q\geq 1$.
\end{lemma}

\begin{proof}
  It is well-known that for all $D\geq 2$ we have the asymptotic equality
  \begin{equation*}
    \sum_{\substack{\q\in\ZZ^D \\ \abs{\q} \leq q \\ \gcd(\q)=1}} 1 \sim 2^D \zeta(D)^{-1}q^D
  \end{equation*}
  as $q \to \infty$, where $\zeta$ denotes the Riemann zeta
  function. (See, for example,~\cite[Lemma~4.2]{DLN}.) In particular,
  there is some $Q_D$ such that
  \begin{equation*}
    \sum_{\substack{\q\in\ZZ^D \\ \abs{\q} \leq q \\ \gcd(\q)=1}} 1 \geq 2^{D-1} \zeta(D)^{-1}q^D
  \end{equation*}
  holds for all $q\geq Q_D$. We may take $C$ as the minimum of the finite set
  \begin{equation*}
    \set*{2^{D-1} \zeta(D)^{-1}}\cup    \set*{q^{-D}\sum_{\substack{\q\in\ZZ^D \\ \abs{\q} \leq q \\ \gcd(\q)=1}} 1: q\leq Q_D}.
  \end{equation*}
  This proves the lemma.
\end{proof}

\begin{lemma}\label{lem:primitivesphere}
  For each integer $D\geq 1$ there exists a constant $C>0$ such that
  \begin{equation*}
    \sum_{\substack{\q\in\ZZ^D \\ \abs{\q} \leq q \\ \gcd(\q, q)=1}} 1 \geq
    \begin{cases}
      C \varphi(q) &\textrm{if } D=1,\\
      Cq^D &\textrm{if } D\geq 2 
    \end{cases}
  \end{equation*}
  for all $q\geq 1$.
\end{lemma}

\begin{proof}
  The case $D=1$ is exactly the definition of $2\varphi(D)$, so let us
  assume $D\geq 2$. Then
  \begin{align*}
    \sum_{\substack{\q\in\ZZ^D \\ \abs{\q} \leq q \\ \gcd(\q, q)=1}} 1 &\geq \sum_{\substack{\q\in\ZZ^D \\ \abs{\q} \leq q \\ \gcd(\q)=1}} 1 \geq C q^D,
  \end{align*}
  for all $q\geq 1$, by Lemma~\ref{lem:primitiveball}.
\end{proof}

\begin{lemma}\label{lem:gcds}
  For each integer $D\geq 1$ there is a constant $C>0$ such that
  \begin{equation*}
    \sum_{\substack{\q \in \ZZ^D \\ \abs{\q} \leq q}}\frac{\varphi(\gcd(\q))}{\gcd(\q)} \geq C q^D\qquad\textrm{and}\qquad  \sum_{\substack{\q \in \ZZ^D \\ \abs{\q} \leq q}}\frac{\varphi(\gcd(\q, q))}{\gcd(\q,q)} \geq C q^D
  \end{equation*}
  both hold for all $q\geq 1$.
\end{lemma}

\begin{proof}
  For the first expression, if $D\geq 2$, we have
  \begin{align*}
    \sum_{\substack{\q \in \ZZ^D \\ \abs{\q} \leq q}}\frac{\varphi(\gcd(\q))}{\gcd(\q)} &\geq \sum_{\substack{\q \in \ZZ^D\, \\ \abs{\q} \leq q \\ \gcd(\q)=1}}\frac{\varphi(\gcd(\q))}{\gcd(\q)} 
    = \sum_{\substack{\q \in \ZZ^D\, \\ \abs{\q} \leq q \\ \gcd(\q)=1}}1 
    \gg q^D, 
  \end{align*}
  where the last inequality follows from
  Lemma~\ref{lem:primitiveball}. On the other hand, if $D=1$, we may
  write
  \begin{align*}
    \sum_{\substack{\q \in \ZZ \\ \abs{\q} \leq q}}\frac{\varphi(\gcd(\q))}{\gcd(\q)}
    &= \sum_{\substack{q' \in \ZZ \\ \abs{q'} \leq q}}\frac{\varphi(q')}{q'}
   \gg q.
  \end{align*}
  The last estimate is well-known. It follows from the
    fact that the average order of $\varphi(q)$ is $6q/\pi^2$ (see,
    for example,~\cite[Theorem~330]{HW}).

  For the second expression in the lemma, if $D\geq 2$, we have
    \begin{align*}
      \sum_{\substack{\q \in \ZZ^D \\ \abs{\q} \leq q}}\frac{\varphi(\gcd(\q,q))}{\gcd(\q,q)}
      &\geq \sum_{\substack{\q \in \ZZ^D\, \\ \abs{\q} \leq q \\ \gcd(\q)=1}}\frac{\varphi(\gcd(\q,q))}{\gcd(\q,q)} 
      = \sum_{\substack{\q \in \ZZ^D\, \\ \abs{\q} \leq q \\ \gcd(\q)=1}} 1 
      \gg q^D,
    \end{align*}
    where the final inequality again follows by
    Lemma~\ref{lem:primitiveball}. If $D=1$, we write

  \begin{align*}
    \sum_{\substack{\q \in \ZZ \\ \abs{\q} \leq q}}\frac{\varphi(\gcd(\q,q))}{\gcd(\q,q)}
    &= \sum_{d\mid q} \frac{\varphi(d)}{d}\sum_{\substack{q' \in \ZZ\, \\ \abs{q'} \leq q \\ \gcd(q',q)=d}} 1 \\
    &= \sum_{d\mid q} \frac{\varphi(d)}{d}\sum_{\substack{q' \in \ZZ\, \\ \abs{q'} \leq q/d \\ \gcd\left(q',\frac{q}{d}\right)=1}} 1 \\
    &= 2 \sum_{d\mid q} \frac{\varphi(d)}{d}\varphi(q/d).
  \end{align*}
  Meanwhile, we have
  \begin{align*}
    \sum_{d\mid q}\frac{\varphi(d)\varphi(q/d)}{d}
    &= \sum_{d\mid q}\varphi(q/d) \sum_{j \mid d}\frac{\mu(j)}{j} \\
    &= \sum_{j \mid q}\frac{\mu(j)}{j} \sum_{i \mid (q/j)}\varphi(q/ij) \\
    &= \sum_{j \mid q}\frac{\mu(j)}{j} \parens*{\frac{q}{j}}\\
    &= q \sum_{j \mid q}\frac{\mu(j)}{j^2} \\
    &= q \prod_{p \mid q}\parens*{1 - p^{-2}}.
  \end{align*}
  Noting that (see \cite[Theorem 280]{HW}) for every $q\geq 1$ we have
  \begin{equation*}
    \zeta(2)^{-1} \leq \prod_{p \mid q}\parens*{1 - p^{-2}} \leq 1,
  \end{equation*}
we see that
\begin{align*}
 \sum_{d\mid q}\frac{\varphi(d)\varphi(q/d)}{d}
    = q \prod_{p \mid q}\parens*{1 - p^{-2}}
    &\asymp q.
\end{align*}
This finishes the proof of the lemma.
\end{proof}

\begin{lemma}\label{lem:funny}
  Suppose $\pi = (\pi_1, \dots, \pi_k)$ is a partition of $[m+n]$ such
  that for every $j=1,\dots,k$ we have
  $\abs{\pi_j}\geq 2$, as in
  Section~\ref{sec:partition-reduction}. Suppose there is some
  $\ell\in\{1,2,\dots,k\}$ such that
  $\pi_\ell\cap(m+[n])\neq\emptyset$ and $\abs{\pi_\ell}\geq 3$. Then
  there exists a constant $C >0$ such that
  \begin{equation*}
    \sum_{\substack{\abs{\q}=q \\ \q \in Q(\pi)}}\prod_{{\substack{1 \leq j \leq k \\\abs{\pi_j\cap[m]}=1}}}\frac{\varphi(\gcd(\pi_j(\q)))}{\gcd(\pi_j(\q))} \geq C q^{n-1}
  \end{equation*}
  for all $q\geq 1$. If no such $\ell$ exists, then the above sum can be bounded below by $\geq C q^{n-2}\varphi(q)$.
\end{lemma}

\begin{proof} 
  Suppose first that $\ell \in (b,k]$, so $\pi_{\ell} \subset
  m+[n]$. By restricting the sum to those $\q$ whose norm is achieved
  in a coordinate corresponding to $\pi_\ell$, i.e.
  $|\pi_{\ell}(\q)| = |\q|$, we may bound
    \begin{equation*}
      \sum_{\substack{\abs{\q}=q \\ \q \in
          Q(\pi)}}\prod_{{\substack{1 \leq j \leq k \\\abs{\pi_j\cap[m]}=1}}}\frac{\varphi(\gcd(\pi_j(\q)))}{\gcd(\pi_j(\q))}
      \geq
      \sum_{\substack{\abs{\pi_\ell(\q)}=q \\ \q \in
          Q(\pi)}}\prod_{{\substack{1 \leq j \leq k \\\abs{\pi_j\cap[m]}=1}}}\frac{\varphi(\gcd(\pi_j(\q)))}{\gcd(\pi_j(\q))}\prod_{{\substack{1 \leq j \leq k \\\abs{\pi_j\cap[m]}\neq 1}}}1.
    \end{equation*}
     We can further split
    \begin{multline*}
      \sum_{\substack{\abs{\pi_\ell(\q)}=q \\ \q \in
          Q(\pi)}}\prod_{{\substack{1 \leq j \leq k \\\abs{\pi_j\cap[m]}=1}}}\frac{\varphi(\gcd(\pi_j(\q)))}{\gcd(\pi_j(\q))}\prod_{{\substack{1 \leq j \leq k \\\abs{\pi_j\cap[m]}\neq 1}}}1 \\
      = \brackets*{\prod_{\abs{\pi_j\cap[m]}=1}
        \sum_{\substack{\q \in \ZZ^{\abs{\pi_j\cap(m+[n])}} \\
            \abs{\q} \leq q}}\frac{\varphi(\gcd(\q))}{\gcd(\q)}}
      \brackets*{\prod_{\abs{\pi_j\cap[m]}>1}
        \sum_{\substack{\q\in\ZZ^{\abs{\pi_j\cap(m+[n])}} \\
            \abs{\q}\leq q}} 1} \\ \times
      \brackets*{\prod_{\substack{j\in (b,k] \\ j\neq \ell}}
        \sum_{\substack{\q\in\ZZ^{\abs{\pi_j}} \\ \abs{\q}\leq q,\,
            \gcd(\q)=1}} 1} \brackets*{\sum_{\substack{\q \in
            \ZZ^{\abs{\pi_\ell}-1} \\ \abs{\q}\leq q,\, \gcd(\q,
            q)=1}} 1}.
  \end{multline*}
  Using Lemmas~\ref{lem:primitiveball},~\ref{lem:primitivesphere},
  and~\ref{lem:gcds}, and the fact that $\abs{\pi_\ell}\geq 3$, we see that
  there is some constant $C>0$ such that we can bound the above
  expression below by
  \begin{align*}
    &\geq C \brackets*{\prod_{\abs{\pi_j\cap[m]}=1}
      q^{\abs{\pi_j\cap(m+[n])}}} \brackets*{\prod_{\abs{\pi_j\cap[m]}>1}
      q^{\abs{\pi_j\cap(m+[n])}}} \brackets*{\prod_{\substack{j\in (b,k] \\
    j\neq \ell}} q^{\abs{\pi_j}}} \brackets*{q^{\abs{\pi_\ell}-1}} \\
    &= C q^{n-1},
  \end{align*}
  as needed.

  Now suppose $\ell \in [1,a]$. If $\abs{\pi_\ell \cap [m]}>1$ then by
  again restricting the sum to those $\q$ having
  $\abs{\q} = \abs{\pi_\ell(\q)}$, we have
  \begin{multline*}
    \sum_{\substack{\abs{\q}=q \\ \q \in
        Q(\pi)}}\prod_{\abs{\pi_j\cap[m]}=1}\frac{\varphi(\gcd(\pi_j(\q)))}{\gcd(\pi_j(\q))}
    \\ \geq \brackets*{\prod_{\abs{\pi_j\cap[m]}=1} \sum_{\substack{\q
          \in \ZZ^{\abs{\pi_j\cap(m+[n])}} \\ \abs{\q} \leq
          q}}\frac{\varphi(\gcd(\q))}{\gcd(\q)}}
    \brackets*{\prod_{\substack{\abs{\pi_j\cap[m]}>1 \\ j\neq \ell}}
      \sum_{\substack{\q\in\ZZ^{\abs{\pi_j\cap(m+[n])}} \\
          \abs{\q}\leq q}} 1} \\ \times
    \brackets*{\sum_{\substack{\q\in\ZZ^{\abs{\pi_\ell\cap(m+[n])}-1} \\
          \abs{\q}\leq q}} 1} \brackets*{\prod_{j\in (b,k] }
      \sum_{\substack{\q\in\ZZ^{\abs{\pi_j}} \\ \abs{\q}\leq q,\,
          \gcd(\q)=1}} 1},
  \end{multline*}
  and by Lemmas~\ref{lem:primitiveball} and~\ref{lem:gcds} we bound below by
  \begin{align*}
    &\sum_{\substack{\abs{\q}=q \\ \q \in
    Q(\pi)}}\prod_{\abs{\pi_j\cap[m]}=1}\frac{\varphi(\gcd(\pi_j(\q)))}{\gcd(\pi_j(\q))}\\
    &\geq C \brackets*{\prod_{\abs{\pi_j\cap[m]}=1}
      q^{\abs{\pi_j\cap(m+[n])}}}
      \brackets*{\prod_{\substack{\abs{\pi_j\cap[m]}>1 \\
    j\neq \ell}} q^{\abs{\pi_j\cap(m+[n])}}}
    \brackets*{q^{\abs{\pi_\ell\cap(m+[n])}-1}} \brackets*{\prod_{j\in
    (b,k] }q^{\abs{\pi_j}}} \\ &= C q^{n-1}.
  \end{align*}
  On the other hand, if $\abs{\pi_\ell \cap [m]}=1$, then again by
  restricting the sum to those $\q$ such that $|\pi_{\ell}(\q)|=|\q|$,
  we have
  \begin{align*}
    \sum_{\substack{\abs{\q}=q \\ \q \in
        Q(\pi)}}\prod_{\abs{\pi_j\cap[m]}=1}&\frac{\varphi(\gcd(\pi_j(\q)))}{\gcd(\pi_j(\q))}
    \\ \geq & \brackets*{\prod_{\substack{\abs{\pi_j\cap[m]}=1 \\ j\neq
          \ell}} \sum_{\substack{\q \in \ZZ^{\abs{\pi_j\cap(m+[n])}}
          \\ \abs{\q} \leq q}}\frac{\varphi(\gcd(\q))}{\gcd(\q)}}
    \brackets*{ \sum_{\substack{\q \in
          \ZZ^{\abs{\pi_\ell\cap(m+[n])}-1} \\ \abs{\q}
          \leq q}}\frac{\varphi(\gcd(\q,q))}{\gcd(\q,q)}}\\
    &\phantom{===================} \times\brackets*{\prod_{\abs{\pi_j\cap[m]}>1}
      \sum_{\substack{\q\in\ZZ^{\abs{\pi_j\cap(m+[n])}} \\
          \abs{\q}\leq q}} 1} \brackets*{\prod_{j\in (b,k] }
      \sum_{\substack{\q\in\ZZ^{\abs{\pi_j}} \\ \abs{\q}\leq q,\,
          \gcd(\q)=1}} 1}\\ \geq &
    C\brackets*{\prod_{\substack{\abs{\pi_j\cap[m]}=1 \\ j\neq \ell}}
      q^{\abs{\pi_j\cap(m+[n])}} }
    \brackets*{q^{\abs{\pi_\ell\cap(m+[n])}-1}}
    \brackets*{\prod_{\abs{\pi_j\cap[m]}>1}q^{\abs{\pi_j\cap(m+[n])}}}
    \brackets*{\prod_{j\in (b,k] }q^{\abs{\pi_j}}} \\ =& C q^{n-1}
  \end{align*}
  for all $q\geq 1$. The penultimate line in this case again follows
  from Lemmas~\ref{lem:primitiveball} and~\ref{lem:gcds}. This proves
  the first part of the lemma.

  Suppose now that there is no $\ell$ such that
  $\pi_\ell\cap(m+[n])\neq\emptyset$ and $\abs{\pi_\ell}\geq 3$. If
  instead there is some $\ell\in (b,k]$ with $|\pi_{\ell}|=2$, then
  restricting the sum to the $\q$ such that
  $\abs{\q} = \abs{\pi_\ell(\q)}$ and using
  Lemmas~\ref{lem:primitiveball}, \ref{lem:primitivesphere},
  and~\ref{lem:gcds}, we find
  \begin{align*}
    \sum_{\substack{\abs{\q}=q \\ \q \in
        Q(\pi)}}\prod_{\abs{\pi_j\cap[m]}=1}&\frac{\varphi(\gcd(\pi_j(\q)))}{\gcd(\pi_j(\q))}
    \\
    \geq & \brackets*{\prod_{\abs{\pi_j\cap[m]}=1} \sum_{\substack{\q
          \in \ZZ^{\abs{\pi_j\cap(m+[n])}} \\ \abs{\q} \leq
          q}}\frac{\varphi(\gcd(\q))}{\gcd(\q)}}
    \brackets*{\prod_{\abs{\pi_j\cap[m]}>1}
      \sum_{\substack{\q\in\ZZ^{\abs{\pi_j\cap(m+[n])}} \\
          \abs{\q}\leq q}} 1} \\ 
          &\phantom{===================}\times
    \brackets*{\prod_{\substack{j\in (b,k] \\ j\neq \ell}}
      \sum_{\substack{\q\in\ZZ^{\abs{\pi_j}} \\ \abs{\q}\leq q,\,
          \gcd(\q)=1}} 1} \brackets*{\sum_{\substack{\q \in
          \ZZ^{\abs{\pi_\ell}-1} \\ \abs{\q}\leq q,\, \gcd(\q, q)=1}}
      1} \\
    \gg & \brackets*{\prod_{\abs{\pi_j\cap[m]}=1} q^{\abs{\pi_j\cap(m+[n])}}}
    \brackets*{\prod_{\abs{\pi_j\cap[m]}>1}q^{\abs{\pi_j\cap(m+[n])}}}
    \brackets*{\prod_{\substack{j\in (b,k] \\ j\neq \ell}} q^{\abs{\pi_j}}}
    \brackets*{\varphi(q)} \\
    \gg & q^{n-2}\varphi(q).
  \end{align*}
  If there is no $\ell\in (b,k]$, then we must have $n=1$ and there
  must necessarily exist some $\ell\in[1,a]$ such that
  $\abs{\pi_\ell \cap [m]}=1$. Restricting the sum again to $\q$ such
  that $|\pi_{\ell}(\q)|=|\q|$, we have
  \begin{align*}
    \sum_{\substack{\abs{\q}=q \\ \q \in
    Q(\pi)}}\prod_{\abs{\pi_j\cap[m]}=1}&\frac{\varphi(\gcd(\pi_j(\q)))}{\gcd(\pi_j(\q))}
    \\ \geq & \brackets*{\prod_{\substack{\abs{\pi_j\cap[m]}=1 \\ j\neq
    \ell}} \sum_{\substack{\q \in \ZZ^{\abs{\pi_j\cap(m+[n])}}
    \\ \abs{\q} \leq q}}\frac{\varphi(\gcd(\q))}{\gcd(\q)}}
    \brackets*{ \sum_{\substack{\q \in
    \ZZ^{\abs{\pi_\ell\cap(m+[n])}-1} \\ \abs{\q}
    \leq q}}\frac{\varphi(\gcd(\q,q))}{\gcd(\q,q)}}\\
                                        &\phantom{===================}\times \brackets*{\prod_{\abs{\pi_j\cap[m]}>1}
                                          \sum_{\substack{\q\in\ZZ^{\abs{\pi_j\cap(m+[n])}} \\
    \abs{\q}\leq q}} 1} \brackets*{\prod_{j\in (b,k] }
    \sum_{\substack{\q\in\ZZ^{\abs{\pi_j}} \\ \abs{\q}\leq q,\,
    \gcd(\q)=1}} 1}\\ 
    \gg & \brackets*{\prod_{\substack{\abs{\pi_j\cap[m]}=1 \\ j\neq \ell}}
    q^{\abs{\pi_j\cap(m+[n])}} }
    \brackets*{\frac{\varphi(q)}{q}}
    \brackets*{\prod_{\abs{\pi_j\cap[m]}>1}q^{\abs{\pi_j\cap(m+[n])}}}
    \brackets*{\prod_{j\in (b,k] }q^{\abs{\pi_j}}} \\
     \gg & q^{n-2}\varphi(q).
  \end{align*}
  The penultimate line in this case follows from
  Lemmas~\ref{lem:primitiveball} and~\ref{lem:gcds}. This completes
  the proof of the lemma.
\end{proof}

\section{Counting}
\label{sec:counting}

The main purpose of this section is to prove Lemma~\ref{lem:counting}
which states that, in a sense, the integer points $(\p,\q)\in P(\pi)$
with $\abs{\q}= q$ are uniformly distributed for large $q$, where
$\pi$ is a partition of $[m+n]$ as in
Section~\ref{sec:partition-reduction}. The lemma is important later in
our proof of Lemma~\ref{lem:uniformity}, where we show that the sets
$A^\pi(\q)$ (defined in the next section) also enjoy a
kind of uniform distribution in $\I^{nm}$ where
$\I = [0,1]$.

 The following lemma can be deduced from~\cite[Lemma~1]{Nied}. We include its proof for completeness.

\begin{lemma}\label{lem:niederreiter}
  For any $0\leq \alpha < \beta \leq 1$ with $\beta-\alpha = \gamma$,
  there exists some integer $Q_\gamma>0$ such that
  \begin{equation}\label{eq:1}
    \frac{1}{2}\varphi(q)\gamma \leq \#\set*{p\in\NN : \gcd(p,q)=1, \quad \alpha q \leq p \leq \beta q} \leq   \frac{3}{2}\varphi(q)\gamma
  \end{equation}
  whenever $q\geq Q_\gamma$.
\end{lemma}

\begin{proof}
  Let $\theta(q) = \#\set*{p/q\in (\alpha, \beta)}$, and notice that
  $\floor{\gamma q} \leq \theta(q)\leq \floor{\gamma q}+1$. We have
  \begin{equation*}
    \theta (q) = \sum_{d\mid q} \#(P_d\cap (\alpha, \beta)),
  \end{equation*}
  where $P_d$ denotes the reduced fractions in $[0,1]$ with
  denominator $d$.  The M{\"o}bius inversion formula (see \cite[Theorem 266]{HW}) gives
  \begin{equation*}
    \#(P_q\cap (\alpha, \beta)) = \sum_{d\mid q} \mu\parens*{\frac{q}{d}}\theta(d)
  \end{equation*}
  For a lower bound, we have
  \begin{align*}
    \#(P_q\cap (\alpha, \beta)) &= \sum_{d\mid q} \mu\parens*{\frac{q}{d}}\theta(d)\\
                    &\geq \sum_{d\mid q} \mu\parens*{\frac{q}{d}}\gamma d -  \sum_{d\mid q} \set{\gamma d} \mu \parens*{\frac{q}{d}}\\
                    &= \gamma \varphi(q)-  \sum_{d\mid q} 1,
  \end{align*}
  and for an upper bound, we have
  \begin{align*}
    \#(P_q\cap (\alpha, \beta)) &= \sum_{d\mid q} \mu\parens*{\frac{q}{d}}\theta(d)\\
                                &\leq \sum_{d\mid q} \mu\parens*{\frac{q}{d}}\gamma d +  \sum_{d\mid q} 1 \\
                                &= \gamma \varphi(q) +  \sum_{d\mid q} 1.
  \end{align*}
  This last term is the number of divisors of $q$, which is
  $o(q^\eps)$ for any $\eps>0$~\cite[Theorem~315]{HW}, and in
  particular $o(\varphi(q))$  (since 
    $\varphi(q) \gg q/\log\log q$ by~\cite[Theorem~328]{HW}).
  Therefore, there exists $Q_\gamma>0$ such that for all
  $q\geq Q_\gamma$ we have
  \begin{equation*}
    \sum_{d\mid q} 1 \leq \frac{1}{2} \gamma \varphi(q).
  \end{equation*}
  Combining this with the previous two bounds gives~(\ref{eq:1}) for
  all $q\geq Q_\gamma$, proving the lemma.
\end{proof}

\begin{lemma}\label{lem:primitivepoints}
Suppose $d \in \N$ and, for $1 \leq i \leq d$, suppose that $\alpha_i, \beta_i \in [0,1]$ are such that
\[0 \leq \alpha_i < \beta_i \leq 1 \quad \text{and} \quad \gamma = \beta_i - \alpha_i \quad \text{ for all } 1 \leq i \leq d.\] 
Then there exist $C, Q_\gamma>0$ such that 
\begin{align} \label{prim bound}
\#\left\{(p_1,\dots,p_d) \in \Z^d: \gcd(p_1,\dots,p_d) = 1 \text{ and } \alpha_iq \leq p_i \leq \beta_iq \quad \text{for } i=1,\dots,d \right\} &\geq C \gamma^d q^d
\end{align}
holds for all $q\geq Q_\gamma$. 
\end{lemma}

\begin{proof}
We proceed by induction, first establishing \eqref{prim bound} in the case when $d=2$. Notice that 
\[\#\left\{(p_1,p_2) \in \Z^2: \gcd(p_1,p_2) = 1 \text{ and } \alpha_iq \leq p_i \leq \beta_iq \quad \text{for } i=1,2 \right\} = \sum_{p_1=\lceil \alpha_1 q \rceil}^{\lfloor \beta_1 q \rfloor}{\sum_{\substack{p_2 = \lceil \alpha_2 q \rceil\\ \gcd(p_1,p_2)=1}}^{\lfloor \beta_2 q \rfloor}{1}}.\]
For a fixed $p_1$, by Lemma \ref{lem:niederreiter}, the inner sum is 
\begin{align*}
\sum_{\substack{p_2 = \lceil \alpha_2 q \rceil\\ \gcd(p_1,p_2)=1}}^{\lfloor \beta_2 q \rfloor}{1} &= \#\left\{p_2 \in \N: \gcd(p_1,p_2)=1 \text{ and } \alpha_2 q \leq p_2 \leq \beta_2 q\right\} \\
                                                 &= \#\left\{p_2 \in \N: \gcd(p_1,p_2)=1 \text{ and } \left(\frac{\alpha_2 q}{p_1}\right)p_1 \leq p_2 \leq \left(\frac{\beta_2 q}{p_1}\right)p_1 \right\} \\
                                                 &\geq \frac{1}{2}\varphi(p_1)\left(\frac{\beta_2 q}{p_1}- \frac{\alpha_2 q}{p_1}\right) \\
                                                 &= \frac{1}{2} \frac{\varphi(p_1)}{p_1} q \gamma.
\end{align*}
Thus, recalling that
$\sum_{n=1}^{N}{\frac{\varphi(n)}{n}} \sim \frac{6}{\pi^2} N$
(see~\cite{HW}), we have that there exists some $C>0, Q>0$ such that
if $q \geq Q$ we have
\[\sum_{p_1=\lceil \alpha_1 q \rceil}^{\lfloor \beta_1 q \rfloor}{\sum_{\substack{p_2 = \lceil \alpha_2 q \rceil\\ \gcd(p_1,p_2)=1}}^{\lfloor \beta_2 q \rfloor}{1}} \geq \frac{q \gamma}{2} \left(\sum_{p_1=\ceil{\alpha_1 q}}^{\lfloor \beta_1 q \rfloor}{\frac{\varphi(p_1)}{p_1}} \right) \geq C q \gamma(\beta_1 q - \alpha_1 q) = Cq^2 \gamma^2.\]
This completes the proof of \eqref{prim bound} in the case that $d=2$. 

Next suppose that \eqref{prim bound} has been established for
$d=k$. We will now show that it also holds when $d=k+1$, and so the
proof is then completed by induction. In the case when $d=k+1$, we are
interested in
\begin{align*}
\mathcal{P}(q):=\left\{(p_1,\dots,p_d,p_{d+1}) \in \Z^{d+1}: \gcd(p_1,\dots,p_d,p_{d+1}) = 1 \text{ and } \alpha_iq \leq p_i \leq \beta_iq \quad \text{for } i=1,\dots,d+1 \right\}.
\end{align*}
However, notice that 
\begin{align*}
\left\{(p_1,\dots,p_d,p_{d+1}) \in \Z^{d+1}: \gcd(p_1,\dots,p_d) = 1 \text{ and } \alpha_iq \leq p_i \leq \beta_iq \quad \text{for } i=1,\dots,d+1 \right\} \subset \mathcal{P}(q).
\end{align*}
Now, by our inductive hypothesis,
\begin{align*}
&\#\left\{(p_1,\dots,p_d,p_{d+1}) \in \Z^{d+1}: \gcd(p_1,\dots,p_d) = 1 \text{ and } \alpha_iq \leq p_i \leq \beta_iq \quad \text{for } i=1,\dots,d+1 \right\} \\
&\phantom{==================}= \sum_{p_{d+1} = \lceil \alpha_{d+1}q \rceil}^{\lfloor \beta_{d+1}q \rfloor}\left(\sum_{p_1 = \lceil \alpha_1 q \rceil}^{\lfloor \beta_1 q \rfloor}{\sum_{p_2 = \lceil \alpha_2 q \rceil}^{\lfloor \beta_2 q \rfloor}{\dots{\sum_{\substack{p_d = \lceil \alpha_d q \rceil \\ \gcd(p_1,\dots,p_d)=1}}^{\lfloor \beta_d q \rfloor}{1}}}}\right) \\
&\phantom{==================}\geq C \sum_{p_{d+1} = \lceil \alpha_{d+1}q \rceil}^{\lfloor \beta_{d+1}q \rfloor}{q^d \gamma^d} \\
&\phantom{==================}\geq C q^d \gamma^d (\beta_{d+1}q-\alpha_{d+1}q) \\
&\phantom{==================}\geq C q^{d+1} \gamma^{d+1}.
\end{align*}
This completes the proof of the lemma.
\end{proof}

\begin{lemma}\label{lem:counting}
  Suppose $\pi = (\pi_1, \dots, \pi_k)$ is a partition of $[m+n]$ such
  that for every $j=1,\dots,k$ we have $\abs{\pi_j}\geq 2$, as in
  Section~\ref{sec:partition-reduction}. Suppose there is some
  $\ell \in \{1,\dots,k\}$ such that
  $\pi_\ell\cap (m+[n])\neq \emptyset$ and $\abs{\pi_\ell}\geq
  3$. Then there exists a constant $C>0$ such that the following
  holds. For every $0<\gamma\leq 1$ there exists $Q_\gamma>0$, such
  that for every choice of $0\leq \alpha_i < \beta_i \leq 1$
  ($i=1, \dots, m$) with $\beta_i-\alpha_i = \gamma$, we have
  \begin{equation*}
    \sum_{\abs{\q}=q}\#\set*{\p \in P(\pi, \q): \forall i\in[m],\quad \alpha_i q \leq p_i \leq \beta_i q} \geq C \gamma^m q^{m+n-1}
  \end{equation*}
  as long as $q\geq Q_\gamma$. If no such $\ell$ exists, then the sum
  above is bounded below by $C\gamma^m q^{m + n-2}\varphi(q)$.
\end{lemma}

\begin{proof}
  Suppose $\q\in Q(\pi)$ with $\abs{\q}=q$. Let us estimate
  \begin{equation*}
    \#\set*{\p \in P(\pi, \q): \forall i\in[m],\quad \alpha_i q \leq p_i \leq \beta_i q}.
  \end{equation*}
  We will do this by analysing our freedom of choice of $\p$ in the
  different components defined by~$\pi$. For $j\in [1, b]$, let
  \begin{equation*}
    N_j(\q) = \#\set*{\p \in \ZZ^{\abs{\pi_j\cap[m]}} : \gcd(\p, \pi_j(\q))=1,\quad \alpha_i q \leq p_i \leq \beta_i q \quad (i=1, \dots, \abs{\pi_j\cap[m]})},
  \end{equation*}
  and note that for $j\in (a,b]$ this is equivalent to
  \begin{equation*}
    N_j(\q) = \#\set*{\p \in \ZZ^{\abs{\pi_j}} : \gcd(\p)=1,\quad \alpha_i q \leq p_i \leq \beta_i q\quad (i=1, \dots, \abs{\pi_j})}.
  \end{equation*}
  Then we have
  \begin{equation*}
    \#\set*{\p \in P(\pi, \q): \forall i\in[m],\quad \alpha_i q \leq p_i \leq \beta_i q} = \prod_{j=1}^bN_j(\q).
  \end{equation*}

  For $j\in (a, b]$, Lemma~\ref{lem:primitivepoints} tells us that there
  exists $C_j, Q_{j,\gamma}>0$ such that
  \begin{equation*}
    N_j(\q) \geq C_jq^{\abs{\pi_j}}\gamma^{\abs{\pi_j}}
  \end{equation*}
  whenever $q\geq Q_{j,\gamma}$. 

  For $j \in [1,a]$, we consider two sub-cases. First, if $\abs{\pi_j\cap[m]}\geq 2$, then
  \begin{equation*}
    N_j(\q) \geq \#\set*{\p \in \ZZ^{\abs{\pi_j\cap[m]}} : \gcd(\p)=1,\quad \alpha_i q \leq p_i \leq \beta_i q \quad (i=1, \dots, \abs{\pi_j\cap[m]})},
  \end{equation*}
  so, as above, Lemma~\ref{lem:primitivepoints} tells us that there exists
  $C_j, Q_{j,\gamma}>0$ such that
  \begin{equation*}
    N_j(\q) \geq C_j \gamma^{\abs{\pi_j\cap [m]}} q^{\abs{\pi_j\cap[m]}}
  \end{equation*}
  whenever $q\geq Q_{j,\gamma}$. Otherwise, we have
  $\abs{\pi_j\cap[m]} = 1$, in which case,
  \begin{align*}
    N_j(\q) &= \#\{p \in \Z: \gcd(p,\pi_j(\q))=1, \quad \alpha_iq \leq p_i \leq \beta_i(q) \quad (i=1,\dots,|\pi_j \cap [m]|)\} \\
            &= \#\{p \in \Z: \gcd(p,\gcd(\pi_j(\q)))=1, \quad \alpha_iq \leq p_i \leq \beta_i(q) \quad (i=1,\dots,|\pi_j \cap [m]|)\} \\
            &= \#\left\{p \in \Z: \gcd(p,\gcd(\pi_j(\q)))=1, \phantom{\left(\frac{\alpha_iq}{\gcd(\pi_j(\q)}\right)}\right.\\
            &\phantom{=======}\left. \left(\frac{\alpha_iq}{\gcd(\pi_j(\q))}\right)\gcd(\pi_j(\q)) \leq p_i \leq \left(\frac{\beta_iq}{\gcd(\pi_j(\q))}\right)\gcd(\pi_j(\q)) \quad (i=1,\dots,|\pi_j \cap [m]|)\right\}.
\end{align*}
Thus, by Lemma~\ref{lem:niederreiter}, there exists
$C_j, Q_{j,\gamma}>0$ such that
  \begin{equation*}
    N_j(\q) \geq C_j \gamma q \frac{\varphi(\gcd(\pi_j(\q)))}{\gcd(\pi_j(\q))}
  \end{equation*}
  for all $q \geq Q_{j,\gamma}$.

  We conclude that there exists $C>0$ and $Q_\gamma>0$ such that
  \begin{equation*}
    \prod_{j=1}^bN_j(\q) \geq C\gamma^m q^m \prod_{\abs{\pi_j\cap[m]}=1}\frac{\varphi(\gcd(\pi_j(\q)))}{\gcd(\pi_j(\q))}
  \end{equation*}
  holds for all $q\geq Q_\gamma$. Therefore, by Lemma~\ref{lem:funny}, 
  \begin{align*}
    \sum_{\substack{\abs{\q} = q \\ \q \in Q(\pi)}} \prod_{j=1}^bN_j(\q)
    &\geq C\gamma^m q^m \sum_{\substack{\abs{\q} = q \\ \q \in Q(\pi)}} \prod_{\abs{\pi_j\cap[m]}=1}\frac{\varphi(\gcd(\pi_j(\q)))}{\gcd(\pi_j(\q))}\\[1ex]
    &\geq
      \begin{cases}
        C \gamma^m q^{m + n-1} &\textrm{if $\ell$ exists,} \\[1ex]
        C \gamma^m q^{m + n-2} \varphi(q) &\textrm{if not,}
      \end{cases}
  \end{align*}
and the lemma is proved. 
\end{proof}

\section{Uniformity}

Suppose $n,m\in\NN$. For $\q \in \ZZ^n$ and a ball
  $B \subset \RR^m$, let
\begin{equation*}
  A_{n,m}(\q, B) = A(\q, B)= \set*{\bx \in \I^{nm} : \exists \p\in\ZZ^m,\quad \q\bx +\p \in B}
\end{equation*}
and
\begin{equation*}
  A_{n,m}^\pi(\q, B) =  A^\pi(\q, B) = \set*{\bx\in \I^{nm} : \exists \p\in\ZZ^m,\quad (\p,\q) \in P(\pi), \quad \q\bx +\p \in B}
\end{equation*}
and
\begin{equation*}
  A_{n,m}'(\q, B) = A'(\q, B) = \set*{\bx\in \I^{nm} : \exists \p\in\ZZ^m, \forall i\in[m],\quad \gcd(p_i,\q)=1, \quad \q\bx +\p \in B}.
\end{equation*}
Subscripts will be dropped when the context is clear.  Notice that
\begin{equation}\label{eq:Ameasure}
  \abs{A(\q, B)} = \min\set{ \abs{B}, 1}.
\end{equation}
Suppose $(B_q)_{q\in\NN}$ is a
sequence of balls in $\RR^m$. For $\q \in \ZZ^n$, we will write
\begin{equation*}
  A(\q) = A(\q, B_{\abs{\q}}), \quad
  A^\pi(\q) = A^\pi(\q, B_{\abs{\q}}),\quad \text{and} \quad
  A'(\q) = A'(\q, B_{\abs{\q}}). 
\end{equation*}
Theorems~\ref{cor:monotone},~\ref{cor:nonmonotone}, and~\ref{cor:idsc}
can be phrased in terms of sets of the type we have just defined. It
is therefore important for us to establish some lemmas regarding the
measures of these sets.

The following lemmas show that for partitions $\pi$ as in our main
theorem statements, $A^\pi$ and $A$ have comparable measures, even
when intersected with arbitrary open sets, and that the same goes for
$A'$ and $A$.

\begin{lemma}\label{lem:uniformity}
  Suppose $\pi = (\pi_1, \dots, \pi_k)$ is a partition of $[m+n]$ such
  that for every $j=1,\dots,k$ we have
  $\abs{\pi_j}\geq 2$, as in Section~\ref{sec:partition-reduction}.
  \begin{enumerate}[a)]
  \item Suppose there is some $\ell \in \{1,\dots, k\}$
    such that $\pi_\ell\cap (m+[n])\neq \emptyset$ and
    $\abs{\pi_\ell}\geq 3$.  Then there exists a constant
    $C:=C_{\pi}>0$ such that for every open set $U\subset \I^{nm}$
    there is some $Q_U>0$ such that for all $q \geq Q_U$ we have
    \begin{equation*}
      \sum_{\abs{\q}=q}\abs*{A^\pi(\q, B) \cap U} \geq C \sum_{\abs{\q}=q}\abs*{A(\q,B)}\abs*{U}
    \end{equation*}
    for every ball $B\subset \RR^m$.
    
  \item If, on the other hand, no such
    $\ell$ exists, but the measures $\abs{B_q}$ are non-increasing, then
    \begin{equation*}
      \sum_{\abs{\q}\leq Q}\abs*{A^\pi(\q, B_{\abs{\q}}) \cap U} \geq C \sum_{\abs{\q}\leq Q}\abs*{A(\q,B_{\abs{\q}})}\abs*{U}
    \end{equation*}
    for all $Q$ sufficiently large.
  \end{enumerate}
\end{lemma}

\begin{lemma}\label{lem:uniformitydsc}
  Let $n>2$. There exists a constant $C>0$ such that for every open set
  $U\subset \I^{nm}$ there is some $Q_U>0$ such that for all
  $q \geq Q_U$ we have
  \begin{equation*}
    \sum_{\abs{\q}=q}\abs*{A'(\q, B) \cap U} \geq C \sum_{\abs{\q}=q}\abs*{A(\q,B)}\abs*{U}
  \end{equation*}
  for every ball $B\subset \RR^m$.
\end{lemma}

\begin{proof}[Proof of Lemma~\ref{lem:uniformity}]
  First, find a finite union $V$ of disjoint balls contained in $U$
  such that $\abs{V}\geq \abs{U}/2$. (In principle, we could get as
  close to the measure of $U$ as we want.) Without loss of generality,
  we assume all the balls in $V$ have the same radius, $\gamma>0$.

  Now let $W\subset \I^{nm}$ be any ball of radius $\gamma$. It will be
  enough to show that
  \begin{equation} \label{eq:mixingballssubcase}
    \sum_{\abs{\q}=q}\abs*{A^\pi(\q, B) \cap W} \geq C' \sum_{\abs{\q}=q}\abs*{A (\q,B)}\abs*{W}
  \end{equation}
  for all $q\geq Q_\gamma$, where $C'>0$ is some absolute constant which
  may depend on $n,m$, and $\pi$, and $Q_\gamma$ only depends on
  $\gamma$. Importantly for us, $Q_\gamma$ does not depend on $B$ and so, given
  \eqref{eq:mixingballssubcase}, one can deduce the lemma with
  $C = C'/2$.

  Let us write $\x \in \I^{nm}$ as $\x = (\x_1, \x_2,\cdots, \x_m)$
  where $\x_j$ are column vectors. Then for any $\q\in Q(\pi)\subseteq \ZZ^n$ we have
  that
  \begin{equation*}
    \q\x = (\q\cdot\X_1, \dots, \q\cdot \X_a, \q\cdot \X_{a+1}, \dots, \q\cdot \X_b) \in \I^m,
  \end{equation*}
  where $\X_1, \dots, \X_a$ are the projections to the components
  corresponding to $\pi_1\cap[m], \dots, \pi_a\cap[m]$, respectively,
  and $\X_{a+1}, \dots, \X_b$ are the projections to the components
  corresponding to $\pi_{a+1}, \dots, \pi_b$, respectively. For
  $j \in [1,a]$, $\X_j$ is an $n\times \abs{\pi_j\cap[m]}$-matrix, and
  for $j\in(a,b]$, $\X_j$ is an $n\times \abs{\pi_j}$-matrix.

  The condition that
  \begin{equation*}
    \q\x + \p\in B\qquad\textrm{for}\qquad\p=(p_1, \dots, p_m)\in P(\pi,\q)
  \end{equation*}
  is equivalent to $\q\cdot \X_j + \pi_j(\p) \in \pi_j(B)$ for each $j$. Therefore, we
  have
  \begin{equation}\label{eq:product}
    A_{n,m}^\pi(\q,B)\cap W = \parens*{\prod_{j=1}^b A_{n,\abs{\pi_j\cap[m]}}^\pi(\q, \pi_j(B))}\cap W =  \prod_{j=1}^b A_{n,\abs{\pi_j\cap[m]}}^\pi(\q, \pi_j(B))\cap W_j, 
  \end{equation}
  noting that when $j\in (a,b]$ we have $\pi_j\cap[m]=\pi_j$, and
  where $W_j$ is the projection of $W$ to the copy of
  $\I^{n(\abs{\pi_j\cap[m]})}$ corresponding to the $j$th component of
  the above product.   

  Let us now find the ($n(\abs{\pi_j\cap [m]})$-dimensional Lebesgue)
  measure of
  \begin{equation*}
    A_{n,\abs{\pi_j\cap[m]}}^\pi(\q, \pi_j(B))\cap W_j.
  \end{equation*}
  We will accomplish the task in one fell swoop, but it is worth
  bearing in mind that there are two interpretations of what follows,
  depending on whether $j\in [1,a]$ or $j\in (a,b]$. In the latter case,
  we have $\pi_j\cap[m] = \pi_j$ and the condition
  $\gcd(\p,\pi_j(\q))$ is the same as $\gcd(\p)=1$.
  
  Proceeding, we have
  \begin{equation*}
    A_{n,\abs{\pi_j\cap[m]}}^\pi(\q, \pi_j(B))\cap W_j = \set*{\x \in \I^{n\abs{\pi_j\cap[m]}} : \exists \p\in\ZZ^{\abs{\pi_j\cap[m]}}, \quad \gcd(\p,\pi_j(\q))=1, \quad \q\x - \p \in \pi_j(B)},
  \end{equation*}
  having noted that in this case the condition $\p\in P(\pi_j, \q)$ is
  $\gcd(\p, \pi_j(\q))=1$.  Now we express points
  $\x\in\I^{n\abs{\pi_j\cap[m]}}$ as matrices with columns
  $\x = (\x_1, \x_2, \dots, \x_{\abs{\pi_j\cap[m]}})$ (having relabelled
  coordinates), so that the condition of interest is that
  $\q\cdot \x_i - p_i$ lies in an interval of side-length $2r$, the
  projection of $W_j$ to the $i$th coordinate.  Suppose that
  $\abs{\q}$ is achieved in the first coordinate, $q_1=q$. For any
  $\bz\in\I^{(n-1)\abs{\pi_j\cap[m]}}$ (representing the coordinates in the
  rows $2, \dots, n$ of $\I^{n\abs{\pi_j\cap[m]}}$), let
  \begin{equation*}
    S_\bz = \parens*{A_{n,\abs{\pi_j\cap[m]}}^\pi(\q, \pi_j(B))\cap W_j}_\bz = A_{n,\abs{\pi_j\cap[m]}}^\pi(\q, \pi_j(B))_\bz\cap (W_j)_\bz
  \end{equation*}
  be the cross-section through $\bz$ parallel to the
  $\abs{\pi_j\cap[m]}$-dimensional space spanned by the coordinates in the
  first row. Then
  \begin{align*}
    \abs*{A_{n,\abs{\pi_j\cap[m]}}^\pi(\q, \pi_j(B))\cap W_j} &= \int_{\I^{(n-1)\abs{\pi_j\cap[m]}}}\abs{S_\bz}\,d\bz \\
    &= \int_{Y_j}\abs{S_\bz}\,d\bz
  \end{align*}
  where $Y_j$ is the projection of $W_j$ to the last $n-1$ rows'
  coordinates. Meanwhile, $(W_j)_\bz$ is a $\abs{\pi_j\cap[m]}$-dimensional
  ball and
  \begin{align*}
    A_{n,\abs{\pi_j\cap[m]}}^\pi(\q, \pi_j(B))_\bz
    &= \set*{\x \in \I^{\abs{\pi_j\cap[m]}} : \exists \p\in \ZZ^{\abs{\pi_j\cap[m]}}, \quad \gcd(\p,\pi_j(\q))=1, \quad \q\binom{\x}{\bz} - \p \in \pi_j(B)}\\
    &= \set*{\x \in \I^{\abs{\pi_j\cap[m]}} : \exists \p\in \ZZ^{\abs{\pi_j\cap[m]}}, \quad \gcd(\p, \pi_j(\q))=1, \quad q\x + \q\binom{\0}{\bz} - \p \in \pi_j(B)}.
  \end{align*}
  It is a union of disjoint balls of diameter
  $\frac{\abs{B}^{1/m}}{q}$ with centers at the points $\frac{\p}{q}$
  such that $\gcd(\p, \pi_j(\q))=1$ in
  $(\I^{nm})_{\bz} \cong \I^{\abs{\pi_j\cap[m]}}$, translated by
  $\q \binom{\0}{\bz} + \operatorname{center}(\pi_j(B))$. Let
  $N_j(\q, W)$ be the number of such center points which are also
  contained in the ball $\frac{1}{2}W_j$ in $\I^{\abs{\pi_j\cap[m]}}$,
  that is,
  \begin{equation*}
    N_j(\q, W) = \#\set*{\p \in \ZZ^{\abs{\pi_j\cap[m]}} : \gcd(\p, \pi_j(\q))=1,\quad \p/q \in \frac{1}{2}W_j + t},
  \end{equation*}
  where $t$ is the translation vector as above; its precise value
  does not matter.  The reason for bringing our attention to the
  shrunken ball $\frac{1}{2}W_j$ is that each relevant $\p/q$ is the
  center of a diameter-$\abs{B}^{1/m}/q$ sub-ball of
  $A_{n,\abs{\pi_j\cap[m]}}^\pi(\q, \pi_j(B))_\bz$ which is fully contained in
  $W_j$. We can therefore bound
  \begin{equation*}
   \abs{S_\bz} \geq N_j(\q, W) \frac{\abs{B}^{\abs{\pi_j\cap[m]}/m}}{q^{\abs{\pi_j\cap[m]}}} 
  \end{equation*}
  and
  \begin{align*}
    \abs*{A_{n,\abs{\pi_j\cap[m]}}^\pi(\q, \pi_j(B))\cap W_j} &= \int_{Y_j}\abs{S_\bz}\,d\bz \\
                                                         &\geq N_j(\q, W)  \frac{\abs{B}^{\abs{\pi_j\cap[m]}/m}}{q^{\abs{\pi_j\cap[m]}}}\int_{Y_j} d\bz \\
                                                         &= N_j(\q, W)  \frac{\abs{B}^{\abs{\pi_j\cap[m]}/m}}{q^{\abs{\pi_j\cap[m]}}}\abs{Y_j} \\
                                                         &= N_j(\q, W)  \frac{\abs{B}^{\abs{\pi_j\cap[m]}/m}}{q^{\abs{\pi_j\cap[m]}}}(2\gamma)^{(n-1)\abs{\pi_j\cap[m]}}
  \end{align*}
  Note that the argument in this paragraph did not depend on the
  supposition that $\abs{\q}$ was achieved in the first of the $n$
  coordinates, so the measure calculation would have come out the same
  regardless.
  
  Now, combining~(\ref{eq:product}) with the calculation above, we have
  \begin{align*}
    \abs*{A^\pi(\q,B)\cap W} &= \prod_{j=1}^{b}\abs*{{A_{n,|\pi_j \cap [m]|}^{\pi}(\q,\pi_j(B)}) \cap W_j} \\
                             &\geq \frac{\abs{B}}{q^{m}}(2\gamma)^{m(n-1)} \prod_{j=1}^b N_j(\q, W).
  \end{align*}
  Finally,  Lemma~\ref{lem:counting} tells us that
  \begin{equation*}
    \sum_{\abs{\q}=q}\prod_{j=1}^b N_j(\q, W) \geq
    \begin{cases}
      C \gamma^m q^{m+n-1} \\[1ex]
      C \gamma^m q^{m+n-2}\varphi(q),
    \end{cases}
  \end{equation*}
  as long as $q \geq Q_\gamma$, depending on whether there exists
  $\ell$ as in the lemma's statement. If there does exist such an
  $\ell$, then for $q\geq Q_\gamma$ we will have
  \begin{equation*}
    \sum_{\abs{\q}=q}\abs*{A^\pi(\q,B)\cap W} \geq \bar C \abs{B}\gamma^{mn}q^{n-1}\overset{(\ref{eq:Ameasure})}{\gg} \sum_{\abs{\q}=q}\abs*{A(\q,B)}\abs*{W},
  \end{equation*}
  where $\bar C$ is a constant absorbing $C$ and all of the powers of
  $2$ appearing above. If there does not exist such an~$\ell$, but
  $\abs{B_q}$ is non-increasing, then for large $Q$ we will have
  \begin{align*}
    \sum_{\abs{\q}\leq Q}\abs*{A^\pi(\q,B_{\abs{\q}})\cap W}
    &\geq \sum_{q\leq Q}\bar C \abs{B_q}\gamma^{mn}q^{n-2}\varphi(q) \\
    &\gg  \sum_{q\leq Q} \abs{B_q}\gamma^{mn}q^{n-1}.
  \end{align*}
  The last estimate comes from the fact that the average order of
  $q^{n-2}\varphi(q)$ is $\gg q^{n-1}$ and that $\abs{B_q}$ is
  monotonic. Finally, it follows from \eqref{eq:Ameasure} that
  \[\sum_{\abs{\q}\leq Q}\abs*{A^\pi(\q,B_{\abs{\q}})\cap W} \gg
    \sum_{\abs{\q}\leq Q}\abs*{A(\q,B_{\abs{\q}})}\abs*{W},\] which
  proves the lemma.
\end{proof}

The following proof takes the same steps as the previous one, and only
has been modified to accommodate the definition of $A'(\q)$, which is
different from the definition of $A^\pi(\q)$.

\begin{proof}[Proof of Lemma~\ref{lem:uniformitydsc}]
  As in the previous proof, we find a finite union $V$ of disjoint
  balls contained in $U$ such that $\abs{V}\geq \abs{U}/2$, and we
  assume all the balls in $V$ have the same radius, $\gamma>0$.

  Let $W\subset \I^{nm}$ be any ball of radius $\gamma$. It is
  enough to show that
  \begin{equation} \label{eq:mixingballssubcasedsc}
    \sum_{\abs{\q}=q}\abs*{A'(\q, B) \cap W} \geq C' \sum_{\abs{\q}=q}\abs*{A (\q,B)}\abs*{W}
  \end{equation}
  for all $q\geq Q_\gamma$, where $C'>0$ is some absolute constant which
  may depend on $n,m$, and $Q_\gamma$ only depends on
  $\gamma$. Importantly, $Q_\gamma$ does not depend on $B$ and so,
  given \eqref{eq:mixingballssubcasedsc}, one can deduce the lemma
  with $C = C'/2$.

  Let us write $\x \in \I^{nm}$ as $\x = (\x_1, \x_2,\cdots, \x_m)$
  where $\x_j$ are column vectors. Then for any $\q \subseteq \ZZ^n$ we have
  that
  \begin{equation*}
    \q\x = (\q\cdot\x_1, \dots, \q\cdot \x_m) \in \I^m.
  \end{equation*}

  The condition that
  \begin{equation*}
    \q\x + \p\in B\qquad\textrm{for}\qquad \p\in\ZZ^m, \forall i\in[m],\quad \gcd(p_i,\q)=1
  \end{equation*}
  is equivalent to $\q\cdot \x_i + p_i \in B_i$ for each $i$, where
  $B_i$ is the projection of $B$ to the $i$th coordinate. Therefore,
  we have
  \begin{equation}\label{eq:productdsc}
    A'(\q,B)\cap W = \parens*{\prod_{i=1}^m A_{n,1}'(\q, B_i)}\cap W =  \prod_{i=1}^m A_{n,1}'(\q, B_i)\cap W_i, 
  \end{equation}
  where $W_i$ is the projection of $W$ to the copy of
  $\I^n$ corresponding to the $i$th component of
  the above product.

  Let us now find the ($n$-dimensional Lebesgue) measure of
  $A_{n,1}'(\q, B_i)\cap W_i$. We have
  \begin{equation*}
    A_{n,1}'(\q, B_i) = \set*{\x \in \I^n : \exists p \in\ZZ , \quad \gcd(p,\q)=1, \quad \q\x - p \in B_i}.
  \end{equation*}
  Suppose for now that $\abs{\q}$ is achieved in the first coordinate, $q_1=q$. For any
  $\bz\in\I^{n-1}$ (representing the coordinates
  in the rows $2, \dots, n$ of $\I^n$), let
  \begin{equation*}
    S_\bz = \parens*{A_{n,1}'(\q, B_i)\cap W_i}_\bz = A_{n,1}'(\q, B_i)_\bz\cap (W_i)_\bz
  \end{equation*}
  be the cross-section through $\bz$ parallel to the first coordinate
  axis. Then
  \begin{align*}
    \abs*{A_{n,1}'(\q, B_i)\cap W_i} &= \int_{\I^{n-1}}\abs{S_\bz}\,d\bz \\
    &= \int_{Y_i}\abs{S_\bz}\,d\bz
  \end{align*}
  where $Y_i$ is the projection of $W_i$ to the last $n-1$
  coordinates. Meanwhile, $(W_i)_\bz$ is an interval and
  \begin{align*}
    A_{n,1}'(\q, B_i)_\bz
    &= \set*{x \in \I : \exists p \in \ZZ, \quad \gcd(p,\q)=1, \quad \q\binom{x}{\bz} - p \in B_i}\\
    &= \set*{x \in \I : \exists p\in \ZZ, \quad \gcd(p,\q)=1, \quad qx + \q \binom{0}{\bz} - p \in B_i}
  \end{align*}
  It is a union of disjoint intervals of diameter $\abs{B}^{1/m}/q$ with
  centers at the points $p/q$ such that $\gcd(p, \q)=1$ in
  $(\I^n)_{\bz} \cong \I$, translated by
  $\q \binom{0}{\bz} + \operatorname{center}(B_i)$. Let
  $M_i(\q, W)$ be the number of such center points
  which are also contained in the interval  $\frac{1}{2}W_i$ in
  $\I$, that is, 
  \begin{equation*}
    M_i(\q, W) = \#\set*{p \in \ZZ : \gcd(p, \q)=1,\quad p/q \in \frac{1}{2}W_i + t},
  \end{equation*}
  where $t$ is the translation vector as above.  As before, the reason
  considering the contracted interval $\frac{1}{2}W_i$ is that each
  relevant $p/q$ is the center of a diameter-$\abs{B}^{1/m}/q$
  sub-interval of $A_{n,1}'(\q, B_i)_\bz$ which is fully contained in
  $W_i$. We can therefore bound
  \begin{equation*}
   \abs{S_\bz} \geq M_i(\q, W) \frac{\abs{B}^{1/m}}{q} 
  \end{equation*}
  and
  \begin{align*}
    \abs*{A_{n,1}'(\q, B_i)\cap W_i}
    &= \int_{Y_i}\abs{S_\bz}\,d\bz \\
    &\geq M_i(\q, W)  \frac{\abs{B}^{1/m}}{q}\int_{Y_i} d\bz \\
    &= M_i(\q, W)  \frac{\abs{B}^{1/m}}{q}\abs{Y_i} \\
    &= M_i(\q, W)  \frac{\abs{B}^{1/m}}{q}(2\gamma)^{n-1}
  \end{align*}
  Since the argument in this paragraph did not depend on the
  assumption that $\abs{\q}$ was achieved in the first of the $n$
  coordinates, the measure calculation would have come out the same
  regardless of that assumption.
    
  Now, by~(\ref{eq:productdsc}), we have
  \begin{equation*}
    \abs*{A'(\q,B)\cap W} \geq \frac{\abs{B}}{q^{m}}(2\gamma)^{m(n-1)} \prod_{i=1}^m M_i(\q, W).
  \end{equation*}
  Finally, Lemma~\ref{lem:niederreiter} tells us that
  \begin{equation*}
    \prod_{i=1}^m M_i(\q, W) \geq \frac{1}{2^m} q^m\parens*{\frac{\varphi(\gcd(\q))}{\gcd(\q)}}^m\gamma^m 
  \end{equation*}
  as long as $q \geq Q_\gamma$. Then for $q\geq Q_\gamma$ we have
  \begin{equation}\label{eq:4}
    \sum_{\abs{\q}=q}\abs*{A'(\q,B)\cap W}
\geq \bar C \abs{B}\gamma^{mn}  \sum_{\abs{\q} = q} \parens*{\frac{\varphi(\gcd(\q))}{\gcd(\q)}}^m
\end{equation}
Meanwhile, 
\begin{align}
  \sum_{\abs{\q} = q} \parens*{\frac{\varphi(\gcd(\q))}{\gcd(\q)}}^m
  &\asymp \sum_{\substack{\q \in \ZZ^{n-1} \\ \abs{\q} \leq q}}\parens*{\frac{\varphi(\gcd(\q,q))}{\gcd(\q,q)}}^m \nonumber\\
      &\geq \sum_{\substack{\q \in \ZZ^{n-1}\, \abs{\q} \leq q \\ \gcd(\q)=1}}\parens*{\frac{\varphi(\gcd(\q,q))}{\gcd(\q,q)}}^m \nonumber \\
      &= \sum_{\substack{\q \in \ZZ^{n-1}\, \abs{\q} \leq q \\ \gcd(\q)=1}} 1 \nonumber\\
      &\gg q^{n-1},\label{eq:10}
    \end{align}
    by Lemma~\ref{lem:primitiveball}. Putting this back into~(\ref{eq:4}) gives
  \begin{equation*}
    \sum_{\abs{\q}=q}\abs*{A'(\q,B)\cap W}
    \gg \abs{B}\gamma^{mn}q^{n-1} 
\end{equation*}
with an absolute implicit constant, for all sufficiently large
$q$. This establishes~(\ref{eq:mixingballssubcasedsc}) and proves the
lemma. 
\end{proof}

\section{Proofs of Theorems~\ref{cor:monotone},~\ref{cor:nonmonotone}, and~\ref{cor:idsc}}

  We require two lemmas
  from measure theory:

\begin{lemma}[Divergence Borel--Cantelli Lemma,~{\cite[Lemma~2.3]{Harman}}]\label{lem:borelcantelli}
  Suppose $(X, \mu)$ is a finite measure space and
  $(A_q)_{q \in \N} \subset X$ is a sequence of measurable subsets
  such that $\sum \mu(A_q)=\infty$. Then
  \begin{equation} \label{DBC}
    \mu\parens*{\limsup_{q\to\infty} A_q} \geq \limsup_{Q\to\infty} \frac{\parens*{\sum_{q=1}^Q \mu(A_q)}^2}{\sum_{q,r=1}^Q\mu(A_q\cap A_r)}.
  \end{equation}
\end{lemma}

If the expression on the right-hand side of \eqref{DBC} is strictly
positive, then we say that the sets $(A_q)_{q \in \N}$ are
\emph{quasi-independent on average}.

\begin{lemma}[{\cite[Lemma 6]{BDV}}]\label{lem:lebesguedensity}
  Let $(X,d)$ be a metric space with a finite measure $\mu$ such that
  every open set is $\mu$-measurable. Let $A$ be a Borel subset of $X$
  and let $f:\RR_{\geq 0}\to\RR_{\geq 0}$ be an increasing function with $f(x)\to 0$ as
  $x\to 0$. If for every open set $U\subset X$ we have
  \begin{equation*}
    \mu(A\cap U) \geq f(\mu(U)),
  \end{equation*}
  then $\mu(A) = \mu(X)$.
\end{lemma}

Theorem~\ref{cor:monotone} is a consequence of the following theorem.

\begin{theorem}\label{thm:monotone}
  Let $m,n\in\NN$ be such that $nm>2$. Suppose
  $\pi = \set{\pi_1, \dots, \pi_k}$ is a partition of $[m+n]$ with
  $\abs{\pi_j}\geq 2$ for $j =1,\dots,k$. If
  $\Psi:=(B_q)_{q=1}^\infty\subset \RR^m$ is a sequence of balls such
  that $\abs{B_q}$ is non-increasing and such that
  $\sum q^{n-1}\abs{B_q}$ diverges, then for almost every
  $\x\in \operatorname{Mat}_{n\times m}(\RR)$ there exist infinitely
  many points $(\p,\q)\in P(\pi)$ such that
  \begin{equation}\label{eq:2}
    \q\x - \p \in B_{\abs{\q}}.
  \end{equation}
  Conversely, if $\sum q^{n-1}\abs{B_q}$ converges, then for almost
  every $\x\in \operatorname{Mat}_{n\times m}(\RR)$ there are only
  finitely many $(\p,\q)\in P(\pi)$ such that~(\ref{eq:2}) holds.
\end{theorem}

\begin{proof}[Proof of Theorem~\ref{cor:monotone}]
  If $nm>2$, then this is the case of Theorem~\ref{thm:monotone} where
  the balls $B_q$ are concentric around $\y$.  In the cases
  $(m,n) = (2,1)$ and $(m,n)=(1,2)$, the only possible partition is
  the trivial partition. In the case $(m,n)=(1,2)$, the corollary is
  already contained in~\cite[Chapter 1, Theorem~14]{S},
  and in the case $(m,n)=(2,1)$ it is in the inhomogeneous version of
  Khintchine's theorem due to Schmidt~\cite{Schmidt}. In the case
  $(m,n)=(1,1)$ the result is implied by the inhomogeneous Khintchine
  theorem due to Sz{\" u}sz~\cite{Szusz}.
\end{proof}

Theorem~\ref{cor:nonmonotone} is a consequence of the following. 

\begin{theorem}\label{thm:nonmonotone}
  Let $m,n\in\NN$ be such that $nm>2$. Suppose
  $\pi = \set{\pi_1, \dots, \pi_k}$ is a partition of $[m+n]$ such
  that $\abs{\pi_j}\geq 2$ for $j=1,\dots,k$ and for some
  $\ell \in \{1,\dots,k\}$ we have $\abs{\pi_\ell}\geq 3$ and
  \mbox{$\pi_\ell\cap (m+[n])\neq \emptyset$}. If
  $\Psi:=(B_q)_{q=1}^\infty\subset \RR^m$ is a sequence of balls such
  that $\sum q^{n-1}\abs{B_q}$ diverges, then for almost every
  $\x\in \operatorname{Mat}_{n\times m}(\RR)$ there exist infinitely
  many points $(\p,\q)\in P(\pi)$ such that~(\ref{eq:2}) holds.
  
  Conversely, if $\sum q^{n-1}\abs{B_q}$ converges, then for
  almost every $\x\in \operatorname{Mat}_{n\times m}(\RR)$ there are
  only finitely many $(\p,\q)\in P(\pi)$ such that~(\ref{eq:2}) holds.
\end{theorem}

\begin{proof}[Proof of Theorem~\ref{cor:nonmonotone}]
  This is the concentric version of Theorem~\ref{thm:nonmonotone}.
\end{proof}

Theorem~\ref{cor:idsc} is a consequence of the following.

\begin{theorem}[Inhomogeneous Duffin--Schaeffer conjecture for
  systems of linear forms]\label{thm:idsc}
  Let $m,n\in\NN$ with $n>2$. If
  $\Psi:=(B_q)_{q=1}^\infty\subset \RR^m$ is a sequence of balls such
  that
  \begin{equation}\label{eq:dscsum}
    \sum_{\q \in \ZZ^n} \parens*{\frac{\varphi(\gcd(\q))}{\gcd(\q)}}^m\abs{B_{\abs{\q}}} = \infty,
  \end{equation}
  then for almost every $\x\in \operatorname{Mat}_{n\times m}(\RR)$
  there exist infinitely many points $(\p,\q)\in\ZZ^m\times\ZZ^n$ with
  $\gcd(p_i, \q)=1$ for every $i=1, \dots, m$ and such
  that~(\ref{eq:2}) holds. 
  
Conversely, if the sum in~(\ref{eq:dscsum})
  converges, then for almost every
  $\x\in \operatorname{Mat}_{n\times m}(\RR)$ there are only finitely
  many such $(\p,\q)\in\ZZ^m\times\ZZ^n$.
\end{theorem}

\begin{proof}[Proof of Theorem~\ref{cor:idsc}]
  This is the concentric version of the divergence part of
  Theorem~\ref{thm:idsc}.
\end{proof}

The following proofs of
Theorems~\ref{thm:monotone},~\ref{thm:nonmonotone}, and~\ref{thm:idsc}
are nearly identical.

\subsection{Proof of Theorem~\ref{thm:nonmonotone}}

\begin{proof}[Proof of Theorem~\ref{thm:nonmonotone}] 
  By~\cite[Proposition 2, Theorem 8]{AR} the sets $(A(\q))_{q\in\NN}$
  are quasi-independent on average, meaning that the right-hand side
  of the inequality in Lemma~\ref{lem:borelcantelli} is positive. From
  this and the fact that $A^\pi \subset A$, we get that there is some
  $C>0$ such that
  \begin{equation}\label{eq:5}
    \sum_{1\leq \br\leq \q \leq Q} \abs*{A^\pi(\br)\cap A^\pi(\q)\cap U} \leq C \parens*{\sum_{\abs{\q}=1}^Q \abs*{A(\q)}}^2
  \end{equation}
  for infinitely many $Q$. Meanwhile, we have
    \begin{equation}\label{eq:7}
    \sum_{\abs{\q}=q}\abs*{A^\pi(\q) \cap U} \geq C \sum_{\abs{\q}=q}\abs*{A(\q)}\abs*{U}
  \end{equation}
  for all $q\geq Q_U$ by Part (a) of
  Lemma~\ref{lem:uniformity}. In particular, combining this with~(\ref{eq:Ameasure}) and the fact that
    $\sum q^{n-1}\abs{B_q}$ diverges, we find that
    $\sum_{\q \in \Z^n} \abs{A^\pi(\q)\cap U}$ diverges.

  Combining~(\ref{eq:5}) and~(\ref{eq:7}) and collecting constants, we
  see that there is some constant $C>0$ such that
  \begin{equation*}
    \sum_{1\leq \br\leq \q \leq Q} \abs*{A^\pi(\br)\cap A^\pi(\q)\cap U} \leq \frac{C}{\abs{U}^2} \parens*{\sum_{\abs{\q}=1}^Q \abs*{A^\pi(\q)\cap U}}^2
  \end{equation*}
  holds for infinitely many $Q>0$. Therefore, by Lemma~\ref{lem:borelcantelli},
  \begin{equation*}
    \abs*{\limsup_{\abs{\q}\to\infty} A^\pi(\q)\cap U} \geq \frac{\abs{U}^2}{C}.
  \end{equation*}
  Finally, by Lemma~\ref{lem:lebesguedensity},
  $\limsup_{\abs{\q}\to\infty} A^\pi(\q)$ must have full measure in
  $\I^{nm}$ and the theorem is proved.

  For the convergence part of Theorem~\ref{thm:nonmonotone}, notice
  that convergence of $\sum_{q \in \N} q^{n-1}\abs{B_q}$
  and~(\ref{eq:Ameasure}) imply the convergence of
  $\sum_{\q \in \Z^n} \abs{A(\q)}$. The Borel--Cantelli lemma (see,
  e.g. \cite[Lemma 1.2]{Harman}) immediately tells us that
  $\limsup_{\abs{\q}\to\infty} A(\q)$, and hence also
  $\limsup_{\abs{\q}\to\infty} A^\pi(\q)$, has zero measure.
\end{proof}

\subsection{Proof of Theorem~\ref{thm:monotone}}

\begin{proof}[Proof of Theorem~\ref{thm:monotone}]
  Theorem~\ref{thm:nonmonotone} subsumes Theorem~\ref{thm:monotone}
  when there is some $\ell$ such that $\abs{\pi_\ell}\geq 3$ and
  $\pi_\ell \cap (m+[n])\neq \emptyset$, so let us assume there is no
  such $\ell$, but that the measures $\abs{B_q}$ are decreasing. Note
  that the standing assumption on $\pi$ that all components have at
  least two elements implies that $mn\neq 2$. And the case
  $(m,n)=(1,1)$ is already known (Khintchine's theorem~\cite{KT}). So
  let us assume $nm>2$.

  The proof begins exactly like the proof of
  Theorem~\ref{thm:nonmonotone}, to give that there is some $C>0$ such
  that
  \begin{equation*}
    \sum_{1\leq \br\leq \q \leq Q} \abs*{A^\pi(\br)\cap A^\pi(\q)\cap U} \leq C \parens*{\sum_{\abs{\q}=1}^Q \abs*{A(\q)}}^2
  \end{equation*}
  for infinitely many $Q$. Part (b) of Lemma~\ref{lem:uniformity} tells us that if
  $Q$ is large enough we have
    \begin{equation*}
    \sum_{\abs{\q}\leq Q}\abs*{A^\pi(\q) \cap U} \geq C \sum_{\abs{\q}\leq Q}\abs*{A(\q)}\abs*{U},
  \end{equation*}
which in combination with~(\ref{eq:Ameasure}) and the divergence of
  $\sum q^{n-1}\abs{B_q}$ gives that $\sum_\q \abs{A^\pi(\q)\cap U}$
  diverges.

  We are led as before to
  \begin{equation*}
    \sum_{1\leq \br\leq \q \leq Q} \abs*{A^\pi(\br)\cap A^\pi(\q)\cap U} \leq \frac{C}{\abs{U}^2} \parens*{\sum_{\abs{\q}=1}^Q \abs*{A^\pi(\q)\cap U}}^2
  \end{equation*}
  holding for infinitely many $Q \in \N$. Lemma~\ref{lem:borelcantelli} gives
  \begin{equation*}
    \abs*{\limsup_{\abs{\q}\to\infty} A^\pi(\q)\cap U} \geq \frac{\abs{U}^2}{C},
  \end{equation*}
  and Lemma~\ref{lem:lebesguedensity} tells us that
  $\limsup_{\abs{\q}\to\infty} A^\pi(\q)$ must have full measure in
  $\I^{nm}$, thus proving the theorem.

The convergence part of Theorem~\ref{thm:monotone}
is proved exactly as in Theorem~\ref{thm:nonmonotone}.
\end{proof}

\subsection{Proof of Theorem~\ref{thm:idsc}}

\begin{proof}[Proof of Theorem~\ref{thm:idsc}]
  We have $n>2$. In this case, it is found in~\cite[Proposition 2,
  Theorem 8]{AR} and also~\cite[Theorem~14]{S} that the sets
  $(A(\q))_{q\in\NN}$ are quasi-independent on average. Since
  $A' \subset A$, there is some $C>0$ such that
  \begin{equation}\label{eq:8}
    \sum_{1\leq \br\leq \q \leq Q} \abs*{A'(\br)\cap A'(\q)\cap U} \leq C \parens*{\sum_{\abs{\q}=1}^Q \abs*{A(\q)}}^2
  \end{equation}
  for infinitely many $Q$. Lemma~\ref{lem:uniformitydsc} tells us that
    \begin{equation}\label{eq:9}
    \sum_{\abs{\q}=q}\abs*{A'(\q) \cap U} \geq C \sum_{\abs{\q}=q}\abs*{A(\q)}\abs*{U}
  \end{equation}
  for all $q\geq Q_U$. The divergence condition~(\ref{eq:dscsum})
  and~(\ref{eq:4}) imply that $\sum_\q \abs{A'(\q)\cap U}$ diverges.

  Combining~(\ref{eq:8}) and~(\ref{eq:9}), there is some constant
  $C>0$ such that
  \begin{equation*}
    \sum_{1\leq \br\leq \q \leq Q} \abs*{A'(\br)\cap A'(\q)\cap U} \leq \frac{C}{\abs{U}^2} \parens*{\sum_{\abs{\q}=1}^Q \abs*{A'(\q)\cap U}}^2
  \end{equation*}
  holds for infinitely many $Q \in \N$. Therefore, by
  Lemma~\ref{lem:borelcantelli},
  \begin{equation*}
    \abs*{\limsup_{\abs{\q}\to\infty} A'(\q)\cap U} \geq \frac{\abs{U}^2}{C}.
  \end{equation*}
  Lemma~\ref{lem:lebesguedensity} now guarantees that
  $\limsup_{\abs{\q}\to\infty} A'(\q)$ has full measure in $\I^{nm}$
  and the theorem is proved.

For the convergence part, notice that if the sum
    in~(\ref{eq:dscsum}) converges, then by~(\ref{eq:10}) so does the
    sum $\sum q^{n-1}\abs{B_q}$, and we are back in the situation
    from the proofs of Theorems~\ref{thm:monotone}
    and~\ref{thm:nonmonotone}.
\end{proof}

\bibliographystyle{plain}


\end{document}